\documentclass[11pt]{amsart} 
\usepackage{amssymb}
\usepackage[all]{xy}
\usepackage{stmaryrd}

\theoremstyle{plain}
\newtheorem{theorem}{Theorem}[section]
\newtheorem*{theorem*}{Theorem}
\newtheorem{definition}[theorem]{Definition}
\newtheorem{lemma}[theorem]{Lemma}
\newtheorem{corollary}[theorem]{Corollary}
\newtheorem{proposition}[theorem]{Proposition}
\newtheorem{fact}[theorem]{Fact}

\theoremstyle{remark}
\newtheorem{remark}[theorem]{Remark}

\newtheorem{notation}[theorem]{Notation}
\newtheorem{recall}[theorem]{Reminder}
\newtheorem{question}[theorem]{Question}

\newcommand{\A}{\mathcal A}
\newcommand{\C}{\mathcal C}
\newcommand{\comp}{\mathbb C}
\newcommand{\F}{\mathcal F}

\newcommand{\I}{\mathcal I}
\newcommand{\U}{\mathcal U}

\newcommand{\eps}{\varepsilon}
\newcommand{\ph}{\varphi}
\newcommand{\reel}{\mathbb R}
\newcommand{\nat}{\mathbb N}

\newcommand{\band}{{\rm band\,}}
\renewcommand{\Re}{\mathrm {Re\,}}
\renewcommand{\Im}{\mathrm {Im\,}}

\newcommand{\To}{\mathop{\longrightarrow}}
\begin{document}

\title[New axiomatizable classes of Banach spaces]{ New axiomatizable classes of Banach spaces via disjointness-preserving isometries}
\author[Y. Raynaud]{Yves Raynaud}
\address{\noindent  Institut de Math\'ematiques de Jussieu-Paris Rive Gauche, CNRS, UPMC (Univ.Paris 06), and Univ. Paris-Diderot, 4 place Jussieu, F-75252 Paris Cedex 05, France}
\email{\texttt{yves.raynaud@upmc.fr}}

\begin{abstract} Let $\C$ be an axiomatizable class of order continuous real or complex Banach lattices, that is, this class is closed under isometric vector lattice isomorphisms and ultraproducts, and the complementary class is closed under ultrapowers. We show that if linear isometric embeddings of members of $\C$ in their ultrapowers preserve disjointness, the class $\C^\mathcal B$ of Banach spaces obtained by forgetting the Banach lattice structure is still axiomatizable. Moreover if $\C$ coincides with its ``script class'' $\mathcal{SC}$, so does $\C^\mathcal B$ with $\mathcal{SC}^\mathcal B$. This allows us to give new examples of axiomatizable classes of Banach spaces, namely certain Musielak-Orlicz spaces, Nakano spaces, and mixed norm spaces. 
\end{abstract}

\date{\today.}
\bigskip
\maketitle
  
In this article we investigate some instances where  the axiomatizability of a given class $\C$ of real or complex Banach lattices implies the axiomatizability of the class $\C^\mathcal B$ of underlying Banach spaces. Here we deal with the following definition of   axiomatizability: a class of Banach spaces (resp. Banach lattices) is axiomatizable if it is closed  under surjective linear isometries (resp. isometric vector lattice isomorphisms), ultraproducts and  ultraroots. By ``ultraroot'' of a Banach space $X$ we mean a  Banach space  $Y$ such that $X$ is (linearly isometric to) an ultrapower of $Y$; and we have the analogous definition for Banach lattices (but using isometric vector lattice isomorphisms).  Note that to say that a class is closed under ultraroots is equivalent to saying that the complementary class is closed under ultrapowers. 

Of course the term ``axiomatizable'' refers to the possibility of characterizing the class under consideration by a list of axioms, but giving a precise meaning to this possibility requires controlling in a precise manner the logical form of the axioms. A theorem in first order logic asserts that a class of structures (sets equipped with distinguished relations and functions) is the class of models of a theory, i.e. is characterized by a list of sentences in the appropriate language,  iff it is closed under isomorphisms, ultraproducts, and ultraroots. In the various theories of metric structures the notion of ultraproduct that revealed itself to be fruitful is different from the notion used in pure logic, and the first order logic to which they are relevant is also specific. Two equivalent versions of this logic exist at the present time, namely Henson's logic of positive bounded formulas and approximate satisfaction \cite{HI} and the more recent continuous logic for metric structures \cite{BBHU} where the set of truth values is the segment $[0,1]$.

In this question of transferring axiomatizability from the Banach lattice setting to the Banach space one, there is no problem with closure by ultraproducts, which transfer clearly; the issue is in closure under ultraroots. We exhibit in this article a sufficient condition ensuring that a class of order continuous Banach lattices that is axiomatizable in the Banach lattice language remains axiomatizable in the Banach space language, namely that every member of this class has the following property: any linear isometry into any of its ultrapowers preserves disjointness -- we call this condition ``property DPIU'' (disjointness preserving isometries into ultrapowers). This is the  content of Theorem \ref{axiomatizability1} below.

The content of the paper is the following. Section \ref{section:Preliminaries} consists of preliminaries on Banach lattices and their ultraproducts, with reminders  about complex Banach lattices, monotone convergence properties,  and  the structure of disjointness preserving isometries (in the context of most general Banach lattices). In Section \ref{section:sublattices-u.s.c.} we define the ``sublattices up to a sign change'' as the image of a closed vector sublattice by a modulus preserving linear map. It is well known that the closed vector sublattices of a given Banach lattice are characterized as the closed linear subspaces containing the modulus of each of their elements. Similarly it turns out that the ``sublattices up to a sign change'' of a given order continuous Banach lattice may be characterized by the fact that they are preserved by the action of a certain homogeneous binary function, introduced in \cite{L}, that we call the Lacey function.  The bounded linear maps that preserve disjointness are exactly those preserving the Lacey function. This function and its invariance properties are introduced in the context of most general Banach lattices, order continuity is used only when analyzing the subspaces preserving the Lacey function. In section \ref{section:sublattices-u.s.c.} the cases of real and complex Banach lattices are treated separately, since the real case is quite a bit simpler. In the following sections however no distinction is made between the two cases, except in certain proofs which refer to results of papers in which only the real case is considered. Note also that the axiomatizability of a class of Banach lattices and that of the class of their respective  complexified Banach lattices are equivalent (see Proposition \ref{axiom-complex-vs-real-BL} below).

In section \ref{section:UBL} we then give the main result (Theorem \ref{L-inverse ultraprod}). Let us say that a Banach space $X$ is paved by a family $\mathcal F=(X_i)_{i\in I}$ of its subspaces (or $\F$ is a paving of $X$) if for every $\eps>0$, every finite subset of $X$ lies in the $\eps$-enlargement of some $X_i$, $i\in I$. Equivalently, there exists an ultrafilter $\U$ on $I$ such that  for every $x\in X$, $d(x,X_i)$ converges to zero with respect to $\U$ (such an ultrafilter will be called ``adapted to $\F$''). This is a nice condition for getting a natural embedding of $X$ in the ultraproduct $\prod_{i,\U} X_i$. Note that our concept of pavings is similar to previous ones in the literature (see e.g. \cite{JLS}), but it is more general, since we do not assume that  subspaces in our pavings are finite dimensional, nor that they form an increasing chain or even a directed set.  Now, Theorem \ref{L-inverse ultraprod} tells that if  the ultrafilter $\U$ is adapted to the paving $(X_i)_{i\in I}$ and if the ultraproduct $\prod_{i,\U} X_i$  is linearly isometric to an order continuous Banach lattice $L$ satisfying DPIU, then the space $X$ itself is linearly isometric to a  sublattice of $L$, that is moreover contractively complemented whenever $L$ does not contain $c_0$. In particular a Banach space  ultraroot of an order continuous Banach lattice $L$ with DPIU is linearly isometric to a sublattice of $L$, which turns out to be  a Banach lattice ultraroot of $L$. This gives the theorem of transfer of axiomatizability announced  above.

Theorem \ref{L-inverse ultraprod} has another application. Given a class $\C$ of Banach spaces,  we say that a Banach space $X$ is $\C$-pavable iff for every $\eps>0$ there is a paving of $X$ consisting of $(1+\eps)$-isomorphic copies of members of $\C$. There is an analogous definition of $\C$-pavable Banach lattices when $\C$ is a class of Banach lattices. Then we prove that if $\C$ is a class of order continuous Banach lattices closed under ultraproducts and contractive projections, the class of $\C$-pavable Banach lattices coincides with $\C$. If moreover all members of $\C$ have DPIU, then $\C^\mathcal B$, the class of Banach spaces linearly isometric with a member of $\mathcal C$, coincides with the class of $\C^\mathcal B$-pavable Banach spaces.  

In section \ref{script-C}  we give an application of Theorem \ref{L-inverse ultraprod} to ``script-$\mathcal C$-classes''. If $\mathcal C$ is a class of Banach lattices (resp. spaces) consider the class $\mathcal{SC}$ consisting of Banach lattices (resp. spaces) which are paved almost isometrically by finite dimensional members of   $\mathcal C$ (``script-$\mathcal C$-spaces''). This is an evident generalization of the concept of  $\mathcal L_p$-spaces (in fact, $\mathcal L_{p,1+}$-spaces); another kind of generalization may be found in \cite{Ri}. Assume that the class $\mathcal{C}$ consists of order continuous Banach lattices with the DPIU property, and is closed under ultraproducts. Then if the equation $\mathcal{SC}=\mathcal C$ is verified, it transfers  to the class $\mathcal C^\mathcal B$ of underlying Banach spaces. In the particular case where $\C$ is the class  of $L_p$-spaces, this was proved by H.E. Lacey in the seminar notes \cite{L}, which inspired directly the present work. Lacey's proof relies on the fact that linear isometric embeddings between $L_p$ spaces which are close to be isometric are also close to be disjointness preserving, a fact which is now encoded in our condition DPIU. Note that the conditions  $\mathcal{SC}=\mathcal C$ and $\C$ closed under ultraproducts imply that $\C$ is axiomatizable -- in fact they imply that members of $\C$ are characterized by a certain kind of quantified paving by finite dimensional members of $\C$ (we give in Corollary \ref{axioms} a list of ``informal" axioms characterizing such a class $\C$, which could be the starting point for a formal axiomatization in Henson's logic or in continuous logic).

Section \ref{r-convex-BL} is devoted to very large classes of Banach lattices with DPIU, consisting of isometrically $r$-convex Banach lattices, $r>2$, with stricly monotone norm. In this class disjointness of two elements is norm-determined, i.e. it is expressible in the Banach space language (by a list of sentences in Henson's logic, of closed conditions in continuous logic). This implies clearly that linear isometries between members of this class are disjointness preserving. A nice axiomatizable subclass  (in the Banach lattice language) is the class $\mathcal L_{r,s}$ of exactly $r$-convex and $s$-concave Banach lattices, $2<r\leq s<\infty$, which is thus axiomatizable in the Banach space language; and so is any Banach lattice axiomatizable subclass of $\mathcal L_{r,s}$. Applications are given for the class $\mathcal{B}L_pL_q$ (of bands in Bochner spaces $L_p(L_q)$), where $2<p,q<\infty$, and $\mathcal N_{r,s}$ (of Nakano spaces with exponent function essentially included in $[r,s]$), $2<r<s<\infty$, which are known to be  Banach lattice axiomatizable \cite{HR,P,PR}. That isometries between members of these classes preserve disjointness was essentially known \cite{GR,JKP}. Note that in both cases the dual classes are automatically Banach space axiomatizable too, although their members may not satisfy DPIU.

In section \ref{sec:other} we investigate classes of the kind $\mathcal{B}L_pL_q$ whose members all satisfy DPIU although $p,q$ are not both in $(2,+\infty)$. Clearly this is not possible if $p<q\le 2$ or if $q=2<p$, since in both cases $L_p(L_q)$ embeds isometrically into $L_p$, (if e.g. $L_p$ has a nonatomic part and $L_q$ is separable) and the embedding is not disjointness preserving. The same happens if $p=1$ and $L_q$ has dimension 2. Similarly if $p=2<q$ or $q<p\le 2$ we may have linear isometries from $L_p$ into $L_q$ which do not preserve disjointness, but in these cases $L_p(L_q)$ does not embed in $L_q$ if $L_q$ is not trivial (of dimension 1). Thus we need to control the dimension of the $L_q$ fibers in the spaces of kind  $\mathcal{B}L_pL_q$ under consideration to ensure that they have property DPIU. An extreme way to do it is to assume that these fibers are atomless, equivalently we consider the class $\mathcal{B}L_pL_q^a$ of bands in $L_p(L_q)$ spaces where $L_q$ is atomless. This class was proved to be axiomatizable in the Banach lattice language in \cite{HR-11}. It turns out that,  except if $p<q\le 2$ or if $q=2\le p$, its members satisfy DPIU, so it is axiomatizable in the Banach space language too. At the end of the section we refine this result by considering the classes $\mathcal{B}L_pL_q^{\ge n}$ where the $L_q$-fibers are at least $n$-dimensional (equivalently members of these classes contain $L_p(\ell_q^n)$ as a closed vector sublattice). These classes are   axiomatizable  as Banach  lattices and for $n\ge 2$ ($n\ge 3$ if $p=1$) they are DPIU, and thus axiomatizable  as Banach spaces in the same cases as the preceding $\mathcal{B}L_pL_q^a$ classes.  

We refer to \cite{MN} and \cite{LT} for Banach lattice theory and to \cite{H} for definitions and facts concerning ultraproducts.

From section \ref{section:UBL} on, all the statements in which the real or complex nature of the spaces under consideration is not specified are valid indifferently for real or complex Banach spaces and lattices.

\subsubsection* {Aknowledgements}The author thanks C. W. Henson for his critical reading of a previous version of the manuscript and his suggestions for improving it.

\section{Preliminaries}\label{section:Preliminaries}

In this section we recall some definitions and facts about real and complex Banach lattices, Banach spaces and Banach lattices ultraproducts, axiomatizability, monotone convergence properties in Banach lattices, and on the structure of disjointness preserving operators.

\subsection{Banach lattices}

\subsubsection{Banach lattices (real case)}
A linear subspace $E$ in a normed vector  lattice $X$ is called a vector sublattice if it is closed under operations $\vee$ (max) and $\wedge$ (min); equivalently $E$ is closed under operation modulus. A linear subspace $E$ is an {\it ideal} if it is solid, i. e. $x\in E$ and $|y|\le |x|$ imply $y\in E$. Given a subset $A$ of $X$ we denote by $\mathcal I_0(A)$ and $\mathcal I(A)$ the ideal generated by $A$, resp.  and  the closed ideal generated by $A$. Among the closed ideals are the bands, i. e. sets of the form $B=A^\perp$ (the set of all elements in $X$ that are  disjoint from a given subset $A$). A band $B$ in $E$ is a projection band if $E$ splits as $E=B\oplus B^\perp$.

A linear map $T: X\to Y$ between real normed lattices $X$ and $Y$ is a {\it vector lattice homomorphism} (in short ``lattice homomorphism'') if it preserves the lattice operations $\vee$ and $\wedge$, i.e. $T(x\vee y)=Tx\vee Ty$ and $T(x\wedge y)=Tx\wedge Ty$; equivalently $T$ preserves the modulus operation, i.e. $|Tx|=T|x|$. A vector lattice homomorphism which is isometric will be called a {\it lattice isometry}. A contrario we shall use the expression  {\it linear isometry} for a map between Banach lattices in order to underline that this linear isometric map is not supposed to be a lattice homomorphism. Two Banach spaces, resp. lattices $X,Y$ are said to be isomorphic if there is a bijective linear map, resp. a lattice homomorphism $T: X\to Y$ such that $T$ and $T^{-1}$ are bounded. 

A subset $A$ of a Banach lattice $X$ is order bounded if there is some $u\in X_+$ such that $|x|\le u$ for every $x\in A$. A linear operator $T:X\to Y$ between Banach lattices $X,Y$ is order bounded if it maps order bounded sets onto order bounded sets.

\subsubsection{Complex Banach lattices}\label{subsubsec:BL-complex}

Complex Banach lattices are usually defined as complexifications of real
Banach lattices, see e.g. \cite[\S 2.2]{MN} or \cite[\S 3.2]{AA}. To each real Banach lattice $X$ 
is associated its complexification $X_\mathbb C = X+iX$. For $z=a+ib$ we set 
$\bar z=a-ib$ and $\Re z=a=\frac 12 (z+\bar z)$. Note that $X$ appears as the real 
linear subspace $\Re X_\mathbb C$ of $X_\mathbb C$ (the ``real part'' of $X_\mathbb C$).  

The modulus of $z=a+ib\in X_\mathbb C$ is defined as $|z|= \bigvee\limits_\theta 
(\cos\theta) a+(\sin\theta) b$, (which always exists in $X$ as limit of a Cauchy net of 
finite suprema) and the norm is defined by $\|z\|=\|\,|z|\,\|_X$. An alternative expression of the modulus of $z=a+ib$ may be given using Krivine's functional calculus \cite[Theorem 1.d.1]{LT} \nobreak by $ |z|=(a^2+b^2)^{1/2}=({z\bar z})^{1/2}$.
The modulus map $z\mapsto |z|$ is sublinear, homogeneous, maps $X_\mathbb C$ onto $X_+$ and coincides on the real Banach lattice $X$ with the absolute value. A useful monotonicity property  of the modulus that may be deduced from one or the other of the two preceding formulas is
\begin{align}\label{modulus-monotonicity}
\forall a_1,a_2,b_1,b_2\in \Re X_\mathbb C,\quad |a_1|\le |a_2| \hbox{ and } |b_1|\le |b_2|\Longrightarrow |a_1+ib_1|\le |a_2+ib_2|
\end{align}

By a closed complex sublattice of $X_\mathbb C$ is meant  the complexification $Y_\mathbb C$ of any real closed sublattice $Y$ of $X$. A lattice homomorphism from a complex Banach lattice into a second one is the complexification $T_\mathbb C$ of a real lattice homomorphism $T$ between the corresponding real Banach lattices, that is $T_\mathbb C(a+ib)=Ta+iTb$. Since $T$ preserves finite supremas and the modulus of each element $z=a+ib$ is a limit of finite supremas of the form $\bigvee\limits_{j=1}^n (\cos\theta_j) a+(\sin\theta_j) b$, it is easy to see that $T_\mathbb C$ preserves modulus operation, i.e. $|T_\mathbb C z|=T|z|=T_\mathbb C|z|$.
 
 In \cite{HR} a slightly more axiomatic point of view was adopted, where a complex Banach lattice $X$ is defined as a triple $(X,c,|\cdot |)$, where $X$ is a complex Banach space, $c:z\to \bar z$ a norm-preserving antilinear involution on $X$ (the conjugation) and $|\cdot |: X\to X, z\mapsto |z|$ the module map, with the appropriate properties. A lattice homomorphism of complex Banach lattices is a complex linear map intertwining the modulus maps of both lattices. It is easy to see that it is positive and thus preserves the real part, and also the conjugation. Similarly a complex (closed) sublattice of a complex Banach lattice is a closed complex subspace  containing the conjugates and the modulus of any of its elements. In Lemma \ref{cpl-sublat-char} below it is proven that in fact the first hypothesis follows of the second one.

If $a\in X$ we denote by the $\I_0(a)$ the (non closed) ideal generated by $a$, i. e. $\I_0(a)=\{z\in X: \exists n\in \mathbb N,  |z|\le n|a|\}$, and by $\I(a)$ its closure. Equipped with the norm $\|z\|_0 =\inf\{t>0: |z|\le t|a|\}$,  $\I_0(a)$ is a complex Banach lattice isometric and lattice isomorphic with a $\mathbb C$-valued $C(K)$-space, in which the element $|a|$ is represented  by the function $1$ \cite[Cor. 3.21]{AA}.

\medskip
\noindent{\it Polar decomposition of elements in a complex Banach lattice.}
 For every element $z$ in a complex Banach lattice $X$ there is a unique modulus preserving complex linear operator $S_z$ on $\I(z)$ (the ``sign of $z$'') such that $z=S_z |z|$. This operator $S_z$ is defined first on the non-closed ideal $\I_0(|z|)$, using the representation of $\I_0(z)$ as a space $C(K_z;\mathbb C)$, as the multiplication operator by $f_z$, where $f_z\in C(K_z)$ represents $z$; then $S_z$ is extended by continuity to $\I(z)$. Note that $|f_z|=1_{K_z}$, and $\bar f_z= f_z^{-1}$, so that  $|S_z|= I_{\I(z)}$ and $\overline{S_z}=S_z^{-1}$.
 
 As for the unicity of $S_z$, assume that $U:\mathcal I(z)\to\mathcal I(z)$ is a modulus preserving linear operator such that $U|z|=z$, then $V=S_z^{-1}U$  is also modulus preserving and $V|z|=|z|$. Let us show that $V$ is the identity of $\mathcal I(z)$, it will follow that $U=S_z$. For $0\le y\le |z|$ we have $|Vy|=|y|=y$ and $V(|z|-y)=|z|-y$, hence
\begin{equation}\label{eq:V}
 Vy+V(|z|-y)= V|z|=|z|= y+(|z|-y)=|Vy|+|V(|z|-y)|
 \end{equation}
 In any complex Banach lattice the equality $a+b=|a|+|b|$ implies $a\ge 0, b\ge 0$, thus (\ref{eq:V}) implies $Vy=|Vy|=y$.  By linearity and continuity, the equality $Vy=y$, so far proved for $0\le y\le |z|$, remains valid for all $y\in\mathcal I(z)$.
 
\begin{remark}\label{rem:sign-in-sublat}
The same proof shows that if $Y$ is a sublattice of $X$ containing $z$, and $U: \mathcal I(z)\cap Y\to\mathcal I(z)$ is a modulus preserving map such that $U|z|=z$, then $Uy=S_zy$ for every $y\in \mathcal I(z)\cap Y$. Thus {\it the sign of $z$ relative to any vector sublattice $Y$ containing $z$ is the restriction or the sign of $z$ relative to the ambient Banach lattice, to the ideal generated by $z$ in  $Y$}.
\end{remark}
\begin{remark}
If $z\in \Re X$, then $f_z$ is real valued, in fact it is the difference of two disjoint indicator functions. It follows that $S_z=P_z-Q_z$ where $P_z,Q_z$ are two complementary band projections on $\I(z)$, and $P_zz=z_+$, $Q_zz=-z_-$.
\end{remark}
 
The following characterization of  closed vector sublattices in a complex Banach lattice (which is probably folklore, although we could not find a reference of it in the literature), will be of crucial use in section \ref{section:sublattices-u.s.c.}, proof of Proposition \ref{char-comp-sublattice-up-to-sgn}.
\begin{lemma}\label{cpl-sublat-char}
A closed complex linear subspace of a complex Banach lattice is a complex vector sublattice iff it contains the moduli of all of its elements.
\end{lemma}
\begin{proof}
Let $E$ be a closed complex subspace of  a complex Banach lattice $X$, such that $E$ contains the modulus of any of its elements. We want to prove that $E$ contains the real parts of all its elements. It will be sufficient to prove that for every $x\in X$, the vector valued map $\reel\to \Re X$, $t\mapsto \big||x|+t x\big|$ has a derivative at $t=0$ and ${d\over dt}||x|+t x|\bigg|_{t=0}= \Re x$. Indeed by hypothesis if $x\in E$ all the quotients $\frac 1t\big[\big||x|+tx\big|-|x|\big]$, $|t|\ne 0$, belong to $E$. Let $x=a+ib$, $a,b\in \Re X$. Since $\big||x|+tx\big|= \big||x|+ta+itb\big|$ and $|b|\le |x|$ we have by monotonicity of the modulus (eq. \ref{modulus-monotonicity})
\[ \big||x|+ta\big| \le \big||x|+tx\big| \le \big||x|+ta+it|x|\big|\]
On the other hand since $a$ is real and  $|a|\le|x|$
\[ |x|+ta \ge |x|-|ta|\ge (1-|t|)|x|\]
hence for $0\le |t|<1$ we have $\big||x|+ta\big|=|x|+ta$ and
\[ \textstyle \big||x|+ta+it|x| \big| \le  \big||x|+ta+i\frac t{1-|t|}\big(|x|+ta\big)\big| = \big| 1+\frac {it}{1-|t|}\big|\big(|x|+ta\big)
\]
We have thus
\[\textstyle |x|+ta\le \big||x|+tx\big| \le  \big(1+\frac{t^2}{(1-|t|)^2}\big)^{1/2}\big(|x|+ta\big)
\]
hence for $t>0$:
\[ a\le \frac 1t\big[\big||x|+tx\big|-|x|\big]\le a+ \frac{t}{2(1-|t|)^2}\big(|x|+ta\big) \To\limits_{t\to 0^+} a
\]
in the norm topology of $X$. (For $t\to 0^-$ the inequalities have to be reversed).
\end{proof}

\subsection{Ultraproducts and axiomatizability}

Let $I$ be a set and $\U$ be an ultrafilter on $I$. If $(E_i)_{i\in I}$ is a family of Banach spaces indexed by $I$, the ultraproduct $\prod_{i,\U} E_i$ (also denoted by $\prod_\U E_i$ if there is no ambiguity with indices) is the quotient of the Banach space
\[\big(\bigoplus_{i\in I}E_i\big)_\infty=\big\{(x_i)\in \prod_{i\in I} E_i:\sup_{i\in I}\|x_i\|_{E_i}<\infty\big\}
\]
(the $\ell_\infty$-direct sum of the family $(E_i)$, equipped with the sup norm $\|(x_i)\|=\sup_{i\in I}\|x_i\|_{E_i}$)  by the closed linear subspace consisting of $\U$-vanishing families
\[\mathcal N= \big\{(x_i): \lim_{i,\U}\|x_i\|=0\big\} \] 
This quotient space equipped with the quotient norm is automatically a Banach space. For $(x_i)_{i\in I}\in \big(\bigoplus_{i\in I}E_i\big)_\infty$ we denote by $[x_i]_\U$ the corresponding element of $\prod_\U E_i$. The quotient norm of this element is given by the simple formula
\[ \|[x_i]_\U\|= \lim_{i,\U} \|x_i\|_{E_i}\]
If all the $E_i$ coincide with the same space $E$, their $\U$-ultraproduct is called the $\U$-ultrapower of $E$ and denoted by $E_\U$. In this case we denote by  $D_E$  the canonical isometric embedding of $E$ into $E_\U$, given by
\[D_E(x)= [(x)]_\U\]
where $(x)$ is the family with unique value $x$.

Let now $(X_i)_{i\in I}$ be a family of real Banach lattices, then  there is a unique vector lattice structure on $\prod_\U X_i$ for which the max and min operations are given by
\[ \xi\vee \eta= [x_i\vee y_i]_\U,\ \xi\wedge \eta= [x_i\wedge y_i]_\U \hbox{ whenever } \xi=[x_i]_\U, \eta=[y_i]_\U\]
The  corresponding positive cone on $\prod_\U X_i$ is  $C=\{\xi\vee 0: \xi\in  \prod_\U X_i\}$ (see \cite[proof of Proposition 3.2]{H}). The absolute value  on $\prod_\U X_i$ is given by
 \[ |\xi|=[|x_i|]_\U \hbox{ whenever } \xi=[x_i]_\U\]
Equipped with this vector lattice structure, $\prod_\U X_i$ is called the Banach lattice ultraproduct of the family $(X_i)$. In the case of an ultrapower, the canonical embedding $D_X: X\to X_\U$ is a lattice isometry.

Let $\mathbb X=\prod_{j,\U}(X_j)_\mathbb C$ be the ultraproduct of the complexifications of the Banach lattices $X_j$, then clearly $\mathbb X=\mathcal X+i\mathcal X$ and this sum is direct  (since the inequality $\|x+iy\|\ge  \|x\|\vee\|y\|$ is true for every  $x,y\in X_j$, for every $j\in J$, it remains true for $x,y\in \mathcal X$). It follows that, as a linear  space,  $\mathbb X$ is the complexification of $\mathcal X$, with real part operation given by $\Re ([z_j]_\U) = [\Re z_j]_\U$. Let us prove that the complex modulus on  $\mathbb X$ is given (as expected) by
\[ \big|[z_j]_\U\big| =[|z_j|]_\U \]
It will follow at once that the identification of $\mathbb X$ with $\left(\mathcal X\right)_\mathbb C$ is norm-preserving.

 Since on each $(X_j)_\mathbb C$ the  modulus map is 1-Lipschitz, we may define unambiguously a map $\tilde m: \mathbb X\to \mathcal X$ by $\tilde m([z_j]_\U)=[|z_j|]_\U$. On the other hand we may approximate the infinite supremum in the definition of the modulus by  finite ones, uniformly on the class of all real Banach lattices, and use the fact that finite suprema ``pass to ultraproducts''. More specifically, set for every $n\in \mathbb N$
 \[ K_n=\{2\pi \frac kn: k=0,1,\dots n-1\} \]
in any real Banach lattice $X$ we may define the  ``lattice term'' $t_n: X\times X\to X$ by
\[t_n(x,y)=\bigvee_{\theta\in K_n} (\cos \theta)\,x+(\sin\theta)\,y\]
We have clearly in any real Banach lattice $X$:
\[(\cos \theta)\,x+(\sin\theta)\,y \le  t_n(x,y)+\frac{2\pi}n (|x|+|y|) \]
for every $x,y\in X$ and $\theta\in [0,2\pi]$ (it is sufficient to check these inequalities when $X=\mathbb R$: see e.g. \cite[p. 66]{MN} or \cite[Corollary 352M]{F}). Thus
\[t_n(x,y)\le |x+iy|\le t_n(x,y)+\frac{2\pi}n (|x|+|y|) \]
When $x=[x_j]_\U$, $y=[y_j]_\U$ belong to $\mathcal X$ we have $t_n(x,y)=[t_n(x_j,y_j)]_\U$, and thus
\[t_n(x,y)\le \tilde m(x+iy)\le t_n(x,y)+\frac{2\pi}n (|x|+|y|)\]
and thus, for all $x,y\in\mathcal X$ and $n\in\mathbb N$
\[ \big| |x+iy|-\tilde m(x+iy)\big|\le \frac{4\pi}n (|x|+|y|)\]
thus $\tilde m(z)=|z|$ for every $z\in \mathbb X$, as was claimed. It follows that $\|z\|_\mathbb X= \|\tilde m(z)\|_\mathcal X= \| |z| \|_\mathcal X= \|z\|_{\mathcal X_\mathbb C}$ and the identification of $\mathbb X$ with $\mathcal X_\mathbb C$ is isometric. We shall say that $\mathbb X$ is the {\it complex Banach lattice ultraproduct} of the family $(X_j)_{j\in J}$.

Recall that the ultraproduct construction is functorial, in the sense that if $(E_i)_{i\in I}$ and $(F_i)_{i\in I}$ are two families of Banach spaces, resp. Banach lattices and $(T_i)_{i\in I}$ is an uniformly bounded family of linear operators(resp. lattice homomorphisms) $T_i:E_i\to F_i$, one may define the ultraproduct operator (resp. homomorphism) $\prod_\U T_i: \prod_\U E_i\to \prod_\U F_i$ by 
\[{\textstyle\prod_\U }T_i |[x_i]_\U:= [Tx_i]_\U\]
and we have a similar definition for the ultrapower $T_\U$ of a single operator $T:E\to F$.

\begin{definition}
A Banach space, resp. Banach lattice $X$ is an {\rm ultraroot} of the Banach space, resp. Banach lattice $Y$ iff $Y$ is linearly, resp. lattice isometric to an ultrapower of $X$. 
\end{definition}

\begin{definition}
A class $\mathcal C$ of (real or complex) Banach spaces, resp. Banach lattices, is {\rm axiomatizable} if it is closed under linear (resp. lattice) surjective isometries and  ultraproducts, and the complementary class to $\C$ is closed under ultrapowers. Equivalently, the class $\mathcal C$ is closed under surjective isometries, ultraproducts, and ultraroots.
\end{definition} 

\begin{proposition}\label{axiom-complex-vs-real-BL}
 A class $\mathcal C$ of real Banach lattices is axiomatizable if and only if the class $\mathcal C_\mathbb C$ of the complexified Banach lattices of $\mathcal C$ is.
\end{proposition}

\begin{proof} 
We note first that two complex Banach lattices are lattice isometric (as complex Banach lattices) iff their real parts are lattice isometric (as real Banach lattices).

Assume that $\mathcal C$ is axiomatizable. Then $\mathcal C_\mathbb C$ is closed under lattice isometries by the preceding remark, and by ultraproducts by  definition. If $L$ is a complex Banach lattice with an ultrapower $L_\U$ belonging to $\mathcal C_\mathbb C$, then $((\Re L)_\U)_\mathbb C$ is lattice isometric to $L_\U$  and thus belongs to  $\mathcal C_\mathbb C$; then its real part  $(\Re L)_\U$ belongs to $\mathcal C$, hence by the hypothesis $\Re L$ belongs to $\mathcal C$ too, which means that $L$ belongs to $\mathcal C_\mathbb C$: the latter class is thus closed under ultraroots, and finally it is axiomatizable.

The proof of the converse implication is similar.
\end{proof}

\subsection{Monotone Convergence}
A real Banach lattice $X$ is said to be {\it order complete} if every upward directed order bounded set in $X$ has a least upper bound in $X$. In this case every band $B$ in $X$ is a projection band.

The Banach lattice $X$ is said to be {\it order continuous} if every decreasing sequence of positive elements converges in $X$ (the limit is then automatically the greatest lower bound of the sequence). Order continuity implies order completeness. In this case, every closed ideal in $X$ is a (projection) band. In contrast with order completeness, order continuity is hereditary by sublattices (see \cite[section 2.4]{MN}).

$X$ is said to be a Kantorovich-Banach space, in short {\it KB-space}  if every norm bounded monotone sequence  in $X$  converges. This is equivalent to $X$ containing no sub-lattice isomorphic to $c_0$ (equivalently, no linear subspace isomorphic to $c_0$). In this case $X$ is a projection band in its bidual (see \cite[section 2.4]{MN} and \cite[Theorem 1.c.4]{LT}).

By definition, a complex Banach space has one of the preceding monotone convergence properties iff its real part has.

Recall that an ultrafilter $\U$ is called {\it countably incomplete} if it exists a sequence $(U_n)$ of elements of $\U$ with empty intersection. Classical examples are the free ultrafilters on countable sets.

\begin{fact}\label{fact:ultra-oc}
Assume that the ultrafilter $\U$ is countably incomplete. Then for any Banach lattice $E$, the following assertions are equivalent: 

i) $E_\U$ is order continuous

ii) $E_\U$ is a KB-space.

iii) $E$ is a super-KB-space, i. e. the finite dimensional spaces $\ell_\infty^n$, $n\ge 1$, do not embed uniformly in $E$  as normed  vector lattices (equivalently, as normed linear spaces).
\end{fact}

\begin{proof}
The implication ii) $\implies$ i) follows from definitions, and the equivalence of ii) and iii) is a standard fact. Let us  prove the implication  i) $\implies$ ii). Assume that $E_\U$ is not a KB-space, and let $(\xi_n)$ be a sequence of normalized positive pairwise disjoint vectors in $E_\U$ which is equivalent to the unit basis of $c_0$. We claim the existence in $E_\U$ of a positive vector $\xi$ majorizing all the vectors $\xi_n, n\ge 1$. This will prove that $E_\U$ is not order continuous since then the sequence $u_n:=\xi-\sum_{k=1}^n\xi_k$ is decreasing but not convergent (not Cauchy). To prove the claim, for each $n$ let $(x_{i,n})_{i\in I}$ be a bounded family of positive elements in $E$ representing $\xi_n$. Since $C:=\sup_n\|\sum_{k=1}^n\xi_k\|<\infty$ we may find $U_n\in\U$ such that $\|\sum_{k=1}^n x_{i,k}\|\le 2C$ for every $i\in U_n$. We may assume that the sequence $(U_n)$ is decreasing and, using the countable incompleteness of $\U$, that  $\bigcap_n U_n=\emptyset$. Then set $x_i=0$ for $i\in I\setminus U_1$, and for every $n\ge 1$, set $x_i=\sum_{k=1}^n x_{i,k}$ for $i\in U_n\setminus U_{n+1}$. Since $\|x_i\|\le 2C$ for all $i\in I$ we may define   $\xi=[x_i]_\U$ in $E_\U$, and we have $\xi\ge \xi_n$ for every $n\ge 1$ since $x_i\ge x_{i,n}$ as soon as $i\in U_n$.
\end{proof}

\begin{remark}l\label{rem:unif-ultra-oc}
In sections \ref{section:UBL} and \ref{script-C} several statements will refer to axiomatizable classes of order continuous Banach lattices. By Fact \ref{fact:ultra-oc}, being closed under ultrapowers, such a class $\C$ consist in fact of super-KB-spaces. Using the closedness under ultraproducts, one may prove more: the members of $\C$ are uniformly super KB, in the sense that for any $\lambda\ge 1$ the set of dimensions  $d$ such that of $\ell_\infty^d$ embeds $\lambda$-isomorphically in some member of $\C$ is finite. Using \cite[Theorems 1.f.12 and 1.f.7]{LT} one may see that the members of the class $\C$ are uniformly $q$-concave for some $q<\infty$.
\end{remark}

\subsection{Polar decomposition of disjointness preserving operators}

Recall that a boun-ded linear operator $T:X\to Y$, where $X$, $Y$ are Banach lattices, is said to have a modulus if for each $x\in X_+$, the supremum $$|T|x:=\sup \{|Ty|: y\in X, |y|\le x\}$$ exists in $Y$. The map $|T|$ extends in a unique way to a positive linear operator $|T|: X\to Y$, the modulus of $T$  (\cite[Theorem 1.3.2]{MN}).

\subsubsection{The real case}
The following fact derives from a well known result first stated by M. Meyer  \cite[Th\'eor\`eme 1.2]{M76}.

\begin{fact}\label{structure-dp-maps}Let $X,Y$ be two  real Banach lattices. Every bounded linear operator $T: X\to Y$ which preserves disjointness  is the difference of two bounded lattice homomorphisms, with disjoint ranges, and has a modulus $|T|$ which is a bounded lattice homomorphism from $X$ to $Y$. If $T$ is isometric, or surjective, so is $|T|$. Moreover  there exists a band projection $P$ of the closed order ideal $\I(|T|X)$ generated by the range of $|T|$, such that $T=U|T|$, with $U=2P-I$. If $T$ is surjective, or if $Y$ is order complete, $P$ extends to a band projection of $Y$. 
\end{fact}

\begin{proof}[Sketch of proof]
By  the proof of \cite[Theorem 3.1.5]{MN}, $T$ is order bounded. It follows then by  \cite[Theorem 3.1.4]{MN} that $T$ is regular and has a  modulus $|T|$  and moreover
$$\forall x\in X_+,\ \ |T|x=|Tx|$$
 In particular if $T$ is surjective then $|T|(X)\supset Y_+$, hence $|T|$ is also surjective. 
Since $|T|$ is positive and disjointness preserving it is a lattice homomorphism. Then
$$\forall x\in X,\ \ \big | |T|\,x\big |= |T|\,|x| =\big |T\,|x|\big |= |Tx_++Tx_-|=|Tx_+-Tx_-|=|Tx|$$
(since  $Tx_+$ is disjoint from $Tx_-$). It follows that $\|\,|T|\,x\|=\|Tx\|$ for all $x\in X$, and thus $|T|$ is  isometric if (and only if) $T$ is. 

The main argument of the proof of Theorem 3.1.4 in \cite{MN} (see also \cite[Lemma 2.39]{AB}) states that
$$\forall x,y\in X_+,\quad  (Tx)_+\perp (Ty)_-$$
Hence the bands $B_+$ and $B_-$ in $Y$ generated respectively  by the sets $\{(Tx)_+:x\in X_+\}$ and $\{(Tx)_-: x\in X_+\}$ are disjoint. 

If we set
\[T_+=\frac 12 (|T|+T) \quad\hbox{and}\quad T_-=\frac 12 (|T|-T)\]
we have $T_+x=(Tx)_+$ and $T_-x=(Tx)_-$ for every $x\in X_+$, thus $T_+$ and $T_-$ are two positive operators with disjoint ranges (respectively included in $B_+$ and $B_-$), which both map disjoint positive elements on disjoint positive elements, hence they  are disjointness preserving and thus lattice homomorphisms.

We have clearly $\I(TE)=\I(|T|E)=\I(T_+E)+\I(T_-E)$, a disjoint sum, thus  $\I(T_+E)$ is a projection band in $\I(|T|E)$.  Let $P$ the associated band projection, then $T_+=P|T|$ and $T_-=(I-P)|T|$, thus
\[ T=T_+-T_- =P|T|-(I-P)|T|=U|T|\]

If $T$ is surjective, then $Y=B_++B_-$, and $B_+$ is a projection band. If $Y$ is order complete every band in $Y$ is a projection band. In both cases,  the band projection in $Y$ associated with $B_+$ extends $P$.
\end{proof}

\begin {remark} We call ``sign-change operator" on a real Banach lattice $Y$, or simply ``sign change'',  an isometry of the form $U=2P-I$, where $P$ is a band projection.  If $Y$ is a Banach ideal of measurable functions then $U$ is a sign-multiplication operator. Clearly every sign change operator preserves the modulus, i. e. $|Ux|=|x|$ for every element $x\in Y$. Conversely every modulus preserving linear operator $U$ is disjointness  preserving, and $|U|$ is the identity, thus by Fact \ref {structure-dp-maps}, $U$ is a sign-change operator.
\end{remark}


\subsubsection{The complex case}

Let us define a complex sign-change on a complex Banach lattice $X$ as a linear map $U:X\to X$ which {\it preserves the modulus}, i.e.$|Ux|=|x|$ for all $x\in X$. Such a map $U$ is clearly disjointness preserving, and injective. In particular $U$ preserves each ideal $\I_0(z)$ and in the representation of $\I_0(z)$ by a space $C(K_z)$, $U$ appears as the multiplier by a function $f_{U|z|}\in C(K_z)$ (representing the element $U|z|$ of $\I_0(z)$). The sign-change $U$ preserves also the closed  ideal $\I(z)$, on which it coincides with the complex sign operator $S_{U|z|}$ (by unicity of the polar decomposition of the element $U|z|$). In particular $U$ maps each ideal  $\mathcal I(z)$ onto itself, so that it is invertible. Like for a sign operator, the inverse  of $U$ is its conjugate  $\overline U$ (defined by $\overline U z= \overline {U\bar z}$).
In the setting of complex Banach lattices, the structure of disjointness preserving bounded linear operators has the following description:

\begin{fact}\label{structure-comp-dp-maps}
If $X,Y$ are complex Banach lattices and $T:X\to Y$ is a complex linear disjointness preserving bounded operator, then $|T|$ exists and satisfies the equation:
\[|T||z|= | |T| z|= |Tz|\]
for all $z\in X$ \cite{M81}. In particular $|T|$ is a bounded  lattice homomorphism, and $|T|$ is surjective, resp. an isometry whenever $T$ is. Moreover there exists a unique complex sign-change $U$ on the closed order ideal $\I(|T|(X)) =\I(T(X))$ generated by the range of $T$ such that $T=U|T|$. If $Y$ is order complete, $U$ may be extended to a complex sign change on $Y$.
\end{fact}

We have only to justify that  the complex linear operator $T$ is order bounded, and then to apply \cite[Theorems 4 and 7]{BB} and (for the last sentence) \cite[Theorem 8]{GH}.   But $T$ may be expressed in terms of two real-linear bounded operators $A$, $B$ by the well known formula
 \[\forall x,y\in \Re X,\ T(x+iy)= (A+iB)(x+iy)=(Ax-By)+i(Ay+Bx)\] 
$A,B$ are disjointness preserving since they depend linearly on $T$ by the formulas
 \[\forall x\in \Re X,\ Ax=\Re (Tx),\ Bx=\Re(-iTx)\]
Hence  they are order bounded by the aforementioned Theorem 3.1.5 in \cite{MN}, and the order boundedness of $T$ is immediate.

\section{Sublattices up to a sign change}\label{section:sublattices-u.s.c.}

We say that a linear subspace $E$ of a Banach lattice $X$ is a vector sublattice ``up to a sign change" if there is a sign change $U$ on $X$ such that $U(E)$ is a vector sublattice of $X$. It follows from Facts \ref{structure-dp-maps},  \ref{structure-comp-dp-maps} that the image of a (real or complex) Banach lattice by a disjointness preserving isometry with values in an  order complete Banach lattice $Y$ is a (closed) vector sublattice of $Y$ up to a sign change. 

We give now an intrinsic characterization of vector sublattices up to a sign change in an order continuous Banach lattice, introduced by Lacey \cite{L} in the context of $L_p$ spaces.

\subsection{The real case}

\begin{notation}[Lacey's $b$ function]\label{Lacey-map}
For $x,y$ elements of a real vector lattice, we set
\[ b(x,y)= |x|\wedge y_+ - |x|\wedge y_-\]
\end{notation}

\begin{lemma}
A bounded linear operator $T:X\to Y$ preserves disjointness iff it preserves the operation $b$, that is  
$T(b(x,y))=b(Tx,Ty)$ for every $x,y\in X$.
\end{lemma}
\begin{proof}
i) Since $|b(x,y)|=|x|\wedge |y|$ it is clear that $x,y$ are disjoint iff $b(x,y)=0$. Thus if $T$ preserves $b$, it preserves disjointness.

\noindent ii) It is clear that if $T:X\to Y$ is a vector lattice homomorphism then $T(b(x,y))=b(Tx,Ty)$ for every $x,y\in X$. Note also that $b(-x,-y)=-b(x,y)$ for all $x,y\in X$. Let now $U=2P-I$ be a sign-change. Since a band projection is a lattice homomorphism, we have:
\begin{align*}
Pb(Ux,Uy)&=b(PU x, PU y)=b(Px,Py)=Pb(x,y)\\
(I-P)b(Ux,Uy)&=b((I-P)Ux,(I-P)Uy)=b((P-I)x,(P-I)y)\\
&=-b((I-P)x,(I-P)y)=-(I-P)b(x,y)
\end{align*}
hence by summation
\[b(Ux,Uy)= Ub(x,y)\]
Since by Fact \ref{structure-dp-maps} every  disjointness preserving bounded linear operator $T$ is a composition of two normed vector lattice homomorphisms and  a sign-change, the lemma is proven.
\end{proof}

\begin{proposition}\label{char-sublattice-up-to-sgn}
Assume that $X$ is an order continuous real Banach lattice. Then for a closed linear subspace $E$ of $X$ the following assertions are equivalent:

i) $E$ is a closed vector sublattice up to a sign change.

ii) The function  $b$ maps $E\times E$ into $E$.
\end{proposition}
\begin{proof}
i)$\implies$ ii): We have $E=U(F)$ where $F$ is a closed vector sublattice of $X$ and $U$ is some sign-change operator on $X$. Clearly $b$ maps $F\times F$ into $F$, then if $x,y\in F$,  
\[ b(Ux,Uy)=U b(x,y)\in U(F)=E\]
since $U$ is a disjointness preserving isometry.

ii) $\implies$ i): Since $X$ is order complete, for every  $y\in X$  the band generated by $y$ is a projection band, and the associated band projection $P_z$ is given by $P_y(v)=\sup\limits_{n\ge 1}(v\wedge n|y|)$, for every $v\in X_+$. By order continuity of $X$ it follows that $v\wedge n|y|\to P_y(v)$ for the norm topology of $X$.  Thus $P_y x= P_y x_{+}-P_y x_{-}=\lim\limits_{n\to \infty} b(ny,x)\in E$, for every $x,y\in E$. Exchanging the roles of  $x$ and $y$, we obtain that $P_{y_+}|x|-P_{y_-}|x|=\lim\limits_{n\to \infty} b(x,ny)\in E$ too. We set $S_y=P_{y_+}-P_{y_-}=P_y(2P_{y_+}-I)$: this is a change of sign operator ``localized on the support of $y$''. We have thus
\begin{align}
&S_y|x|\in E \label{1}\\
&P_y\,x\  \in E \label{2}
\end{align}
for every $x,y\in E$.
Note for further use that $S_y^2=P_y$ and 
\begin{align*}
|S_yx|&= |P_{y_+}x|+|P_{y_-}x| \quad \hbox{[$P_{y_+}x$ and $P_{y_-}x$ are disjoint]}\\
&= P_{y_+}|x|+P_{y_-}|x| \quad \hbox{[band projections are lattice homomorphisms]}\\
&=P_y|x|
\end{align*}
Choose now a maximal disjoint system $(y_\alpha)_{\alpha\in A}$ of non-zero elements in $E$, we claim that the band $B$ generated by the system $(y_\alpha)_{\alpha\in A}$ contains $E$ (and thus equals the band $B_E$ generated by $E$). Indeed if $x\in E$ and $I\subset A$ is a finite subset, and $B_I$ is the band generated by the finite family $(y_\alpha)_{\alpha\in I}$ in $X$ and $P_I$ the corresponding band projection then by formula (\ref{2})
\[P_I x=\sum_{i\in I} P_{y_i} x\in E\]
Let $P_B$ be the band projection onto $B$. By order continuity of $X$, the element $P_Bx$ is in the closure of the $P_I x$, $I\subset A$ finite, and thus belongs to $E$. Then $y:= x-P_B x$ belongs to $E$ and is disjoint of all $y_\alpha$, hence vanishes by the maximality of the system $(y_\alpha)$. Thus $x=P_B x$ for all $x\in E$.

Let $S=\sum\limits_{\alpha\in A} S_{y_\alpha}$, which converges in the strong operator topology of $B(X)$, then by (\ref{1}) $S|x|\in E$ for every $x\in E$. Since $S^2=\sum\limits_{\alpha\in A} S_{y_\alpha}^2=\sum\limits_{\alpha\in A} P_{y_\alpha}=P_B$, and $|x|\in B$ as does $x$, we have $|x|=P_B|x|\in S(E)$. On the other hand for $x\in E$
\[|Sx|=|\sum_\alpha S_{y_\alpha} x|=\sum_\alpha |S_{y_\alpha}x|= \sum_\alpha P_{y_\alpha}|x|=P_B|x|=|x| \]
thus $|Sx|\in S(E)$ for each $x\in E$, which shows that $S(E)$ is a vector sublattice of $X$.

$S$ is a change of sign operator on $P_BX$, to obtain the desired change of sign on the whole of $X$ put simply $U=S+I-P_B$.
\end{proof}

\subsection{The complex case}

\begin{notation}[\it Lacey's function, complex version]\label{comp-Lacey-map} 
 Define then for $x,y\in X$:
 \[b(x,y)= S_y( |x|\wedge |y|)\]
 (the complex sign $S_y$ was defined in subsection \ref{subsubsec:BL-complex}). The reader will notice that this formula defines Lacey's function  in the real case as well. 
\end{notation}

 \begin{lemma}\label{T-formula}
Let $X,Y$ be two complex Banach lattices. If $T:X\to Y$ is a vector lattice homomorphism then for every $x\in X$ we have $S_{Tx}T=TS_x$. If $U: X\to X$ is a complex sign-change then for  every $x\in X$ we have $S_{Ux}=US_x$. Consequently the function $b$ is preserved under disjointness preserving bounded linear maps.
\end{lemma}

\begin{proof}

a)  If $T:E\to F$ is a vector lattice homomorphism it sends $\mathcal I_0(x)$ onto $\mathcal I_0(Tx)$, thus  $S_{Tx}T$ is well defined on $\mathcal I_0(x)$ (and extends by density to $\I(x)$). If $y\in \mathcal I_0(x)$ then
\[ |TS_xy|= T|S_xy|= T|y|= |Ty|\]
In particular $Ty=0\implies TS_xy=0$, hence there is a linear map $U: T\mathcal I_0(x)\to T\mathcal I_0(x)$ such that
\[\forall y\in  \mathcal I_0(x),\ TS_xy=UTy \]
Moreover $|UTy|=|Ty|$, hence $U$ is modulus preserving. We have 
\[U|Tx| = UT|x|=TS_x|x|=Tx\] 
Note that $T\I_0(x)$ is the ideal generated by $Tx$ in the vector sublattice $TX$. 
It follows from Remark \ref{rem:sign-in-sublat} that $U$ coincides with $S_{Tx}$ on the ideal $T\I_0(x)$ and finally
\[ \forall y\in \mathcal I_0(x),\ TS_xy=UTy=S_{Tx}Ty \]
This equality extends by density to all $y\in \mathcal I(x)$.
\medskip
 
 b) If $U: X\to X$ is a complex sign-change map and $x\in X$, then $|Ux|=|x|$ and $\I_0(Ux)=\I_0(x)$. 
Then
\[ S_{Ux}|Ux|=Ux= US_x|x|= US_x|Ux| \]
thus $S_{Ux}=US_x$ by unicity of the polar decomposition.
\medskip
  
 c)  Using successively points b) and a) above,  we have, if $T=U|T|$ is the polar decomposition of a disjointness preserving complex operator $T$ (as given by Fact \ref{structure-comp-dp-maps}) 
 \begin{align*}
 b(Tx,Ty)&=S_{U|T|y}(|Tx|\wedge |Ty|)=US_{|T|y}(|T||x|\wedge |T||y|)\\
 &=US_{|T|y}|T|(|x|\wedge |y|)= U|T|S_y(|x|\wedge |y|)=Tb(x,y)
\end{align*}
\vskip-12pt
\end{proof}

\begin{remark}
That Lacey's function $b$ is preserved under lattice homomorphism results also from the invariance of Krivine functional calculus under lattice homomorphisms, since $b$ may be defined by applying this calculus  to the homogeneous continuous scalar function $b_\comp:\comp \times\comp \to\comp $, $b_\comp(z,w)= (\left|\frac zw\right| \wedge 1)w $ if $w\ne 0$, $b_\comp(z,0)=0$.
\end{remark}

Now we may state the complex version of Proposition \ref{char-sublattice-up-to-sgn}.

\begin{proposition}\label{char-comp-sublattice-up-to-sgn}
Assume that $X$ is an order continuous complex Banach lattice. Then for a closed linear subspace $E$ of $X$ the following assertions are equivalent:

i) $E$ is a complex vector sublattice up to a sign change.

ii) The function  $b$ maps $E\times E$ into $E$.
\end{proposition}

\begin{proof}
i)$\implies$ ii): We have $E=U(F)$ where $F$ is a closed vector sublattice of $X$ and $U$ is some sign-change operator on $X$. Let us distinguish the functions $b$ on $X$ and on $F$ by denoting them $b_X$ and $b_F$. Let $i_F$ be the inclusion operator from $F$ into $X$. Then, since $U\circ i_F$ is disjointness preserving, $ b_X(U\circ i_F x,U\circ i_F y)=U\circ i_F  b_F(x,y)\in U(F)=E$ for all $x,y\in F$.

ii)$\implies$ i): Let $x,y\in E$. We shall prove that equations (1),(2) in the proof of the real case are still true in the complex case.

We have $b(ny,|x|)=|x|\wedge |ny|\to P_y|x|$; then since $P_y|x|$ as well as the $b(ny,|x|)$ belong to $\I(x)$ and $S_x$ is norm  one,  
\begin{align}\label{eq:convergence}
b(ny,x)= S_x b(ny,|x|)\to S_xP_y|x|
\end{align}
Let $Q_y=I-P_y$, then by Lemma \ref{T-formula}
\[Q_yS_xP_y|x|=S_{Q_yx}Q_yP_y|x|=0\]
and similarly we have $P_yS_xQ_y|x|=0$.
It follows that
\[S_xP_y|x|=P_yS_xP_y|x|=P_yS_x|x|=P_yx\]
thus by eq.(\ref{eq:convergence}), $P_yx=\lim_{n\to \infty} b(ny,x)$ belongs to $E$. 

Similarly, since $S_{ny}=S_y$ for every $n\in \mathbb N$, 
\[b(x,ny)=S_{ny}(|x|\wedge n|y|)\To S_yP_y|x| \]
Let us extend $S_y$ to the Banach lattice $X$ by setting $S_yu=S_yP_yu$ for every $u\in X$, it follows that $S_y|x|\in E$ whenever $x,y\in E$, as in the real case.

The end of the proof is similar to that of Proposition \ref{char-sublattice-up-to-sgn}. Choosing a maximal disjoint system $(y_\alpha)_{\alpha\in A}$ of nonzero vectors in $E$ and setting $S=\sum_{\alpha\in A} S_{y_\alpha}$, we have $\overline S S= S \overline S=P_B$, the band projection onto the band generated by $E$, and $S|x|\in E$ for every $x\in E$. It follows that $|x|\in \overline S (E)$ for every $x\in E$ and that the closed complex linear subspace $\overline S (E)$ contains the modulus of each of its elements. By Lemma \ref{cpl-sublat-char}, $\overline S (E)$ is a complex vector sublattice of $X$, and $E=S (\overline S (E))$ is a complex vector sublattice up to a sign change.
 \end{proof}

\section{Ultraroots of Banach lattices}\label{section:UBL}

\subsection{Main result}\label{subsec:MainResult}

\begin{definition}\label{DPI}
We say that a Banach lattice $L$ has property DPIU if every linear isometry of $L$ into any of its ultrapowers preserves disjointness.
\end{definition}

\begin{definition} Let $(X_i)_{i\in I}$ be a family of subsets of a given metric space $X$. We say that $(X_i)_{i\in I}$ is a paving of $X$ if for every finite subset $F$ of $X$ and every $\eps>0$, there is an index $j$ such that $\mathrm{dist} (x,X_j)<\eps$ for each $x\in F$. If $\U$ is an ultrafilter of subsets of $I$, we say that $\U$ is adapted to the family $(X_i)_{i\in I}$ if for every $x\in X$ and $\eps>0$, the set $\{i: \mathrm{dist} (x,X_i)<\eps\}$ belongs to $\mathcal U$ (in other words for every point $x\in X$ the distance from $x$ to $X_i$ converges to zero with respect to $\U$).
\end{definition}

Note that any paving $(X_i)_{i\in I}$ has an adapted ultrafilter.
 Indeed for $F$ a finite subset of  $X$ and $\eps$ a positive real number let $S_{F,\eps}=\{i\in I: F\subset X_i^\eps\}$, where $X_i^\eps:=\{x\in X: \mathrm{dist}(x,X_i)\le\eps\}$ is the $\eps$-enlargement of $X_i$. Then by hypothesis $S_{F,\eps}\ne\emptyset$ and if $F$, $G$ are two finite subsets and $\eps,\delta>0$,  $S_{F,\eps}\cap S_{G,\delta}\supset S_{F\cup G,\eps\wedge \delta}\ne\emptyset$. Then the set $\mathcal F$ consisting of all the subsets of $I$ containing a set among the sets $S_{F,\eps}$ is a filter, and any ultrafilter $\U$ containing $\mathcal F$ is adapted to the family $(X_i)_{i\in I}$. 

If $(X_i)$ is a paving of a Banach space $X$ by linear subspaces, and $\U$ is an adapted ultrafilter, there is a canonical  linear isometry $\Delta$ of $X$ into $\prod_{i,\U} X_i$ which we may define as follows:
if $x\in X$ choose a family $(x_i)\in \prod X_i$ converging to $x$ and let $\Delta(x)=[x_i]_\U$ be the element of $\prod_{i,\U} X_i$ that is defined by this family. Note that this definition is unambiguous since if $(x'_i)$ is another family $\U$-converging to $x$, the differences $(x_i-x'_i)$ $\U$-converge to 0, that is $[x'_i]_\U=[x_i]_\U$. It is clear that $\|\Delta(x)\|=\|x\|$. Moreover the ultraproduct $\tilde\gamma$ of the inclusion maps $\gamma_i: X_i\to X$ is a natural embedding   of the ultraproduct $\prod_{i,\U} X_i$ into the ultrapower $X_\U$, such that $\tilde\gamma\circ \Delta=D_X$, the canonical (diagonal) embedding of $X$ into $X_\U$. Since $D_X$ and $\tilde\gamma$ are linear and the latter one is injective, the map $\Delta$ must be linear.  Let also $Q:\prod_\U X_i\to X^{**}$ be the contractive linear map defined by 
\begin{equation}\label{eq:Q}
Q([x_i]_\U)= w^*-\lim\limits_{i,\U} x_i
\end{equation}
Then clearly $Q\Delta x=x$ for every $x\in X$, i.e. denoting by $i_X$ the canonical embedding from $X$ in its bidual, we have a commutative diagram
 
\[\xymatrix{&\prod_\U X_i\ar^Q[dr]&\\X\ar^{\Delta}[ur]\ar_{i_X}[rr]
&&X^{**}&}\]

It is easy to see that if $X$ is a Banach lattice and $(X_i)$ a paving of $X$ by vector sublattices, the above defined map $\Delta$ becomes a lattice isometry.

Now we state the main Theorem of this section.

\begin{theorem}\label{L-inverse ultraprod} Let $L$ be an order continuous Banach lattice 
satisfying DPIU. Let $X$ be a Banach space paved by a family $(X_i)$ of linear subspaces, and assume that for some adapted ultrafilter, the ultraproduct of $(X_i)$ is linearly isometric to $L$. Then $X$ itself is linearly isometric to a  closed sublattice of $L$. If moreover $L$ is a KB-space (i.e. does not contain $c_0$) then  this  closed sublattice may be chosen contractively complemented in $L$.
\end{theorem}

\begin{proof}
If $T: \prod_\U X_i\to L$ is a given surjective isometry, then $T\Delta$ is a linear isometry from $X$ into $L$. Let $Y=T\Delta X$, which is a closed linear subspace of $L$ and $V: X\to Y, x\mapsto T\Delta x$ be the resulting surjective isometry from $X$ onto $Y$. Set also $Y_i=VX_i$ and $V_i: X_i\to Y_i, x\mapsto V x$ for each index $i$, and let $\widetilde V=\prod_\U V_i$. Then clearly the canonical embedding $\Delta_Y: Y\to \prod_\U Y_i$ verifies $\Delta_Y V={\widetilde V} \Delta$. Let $J\!=\! \widetilde{V}T^{-1}\!: L\to\prod_\U Y_i$, which is a surjective linear isometry from $L$ onto $\prod_\U Y_i$.
Then for every $x\in X$
\[JVx=JT\Delta x= \widetilde{V}\Delta x = \Delta_Y Vx\]
Since $V:X\to Y$ is surjective, it results that $Jy=\Delta_Y y$ for all $y\in Y$, that is, $J:L\to \prod_\U Y_i$ extends $\Delta_Y: Y\to \prod_\U Y_i$.
 
Thus we may suppose w.l.o.g. that $X$ is a closed linear subspace of $L$ and that the canonical embedding $\Delta: X\to \prod_\U X_i$  extends to a linear surjective isometry $J: L\to \prod_\U X_i$.

Since $X_i\subset X\subset L$ we have a natural inclusion $\prod_\U X_i\subset L_\U$. Let us be more formal by naming $j_i:X_i\to L$ the inclusion maps and $\tilde\jmath: \prod_\U X_i\to L_\U$ the ultraproduct map of the $j_i$'s. We denote by $j$ the inclusion map $X\to L$ and by $D_L$ the diagonal map $L\to L_\U$. The composition $S:=\tilde\jmath J$ is a linear isometry of $L$ into $L_\U$ which, by hypothesis, preserves disjointness (but need not to coincide with the diagonal embedding $D_L$).  We have the following commutative diagram

\[\xymatrix{&\prod_\U X_i\ar^{\tilde\jmath}[dr]&\\X\ar^{\Delta}[ur]\ar_{j}[r]\ar_j[rd]& L \ar_J[u]\ar_S[r]&L_\U\\&L\ar_{D_L}[ru]&}\]
Note that
\begin{equation}
\tilde\jmath\big(\prod_\U X_i\big) \cap D_L(L)= \tilde\jmath\circ \Delta(X)=D_L\circ j(X)\label{intersection}
\end{equation}

In other words, viewed in $L_\U$, the intersection of the spaces $\prod_UX_i$ and $L$ is nothing but  $X$. Indeed if $(x_i)$ is a bounded family in $\prod_i X_i$ and $y\in L$ then
\[ [j(x_i)]_\U = D_L y \iff \lim_{i,\U} j(x_i)= y \implies y\in j(X) \]
since $X$ is closed in $L$.

Now we prove that the map $b:L\times L\to L$ (see Notation \ref{Lacey-map}, resp. \ref{comp-Lacey-map} in the complex case) maps $X\times X$ into $X$. This will prove by Prop. \ref{char-sublattice-up-to-sgn} (resp. Prop. \ref{char-comp-sublattice-up-to-sgn}) that $X$ is a vector sublattice of $L$ up to a sign change $U$, and thus is linearly isometric to the vector sublattice $Y=U(X)$ (resp. $\overline U (X)$). 

Since $S$ is disjointness preserving and $D_L$ is a vector lattice isomorphism, we have for $x,y\in X$
\[Sb(jx,jy)= b(Sj x, Sj y)= b(D_L jx, D_L jy)= D_L b(jx,jy) \]
thus $D_L b(jx,jy) =\tilde\jmath\,  (J b(jx,jy))$, 
which implies by (\ref{intersection}) that $b(jx,jy)\in jX$: thus $b(X\times X)\subset X$ as was announced above.

Assume now that the Banach lattice $L$ does not contain $c_0$ as closed linear subspace, then the same is true of $X$. Since moreover $X$ is linearly isometric to a Banach lattice, it must be contractively complemented in its bidual \cite[Th. 1.c.4]{LT}. Let $\pi: X^{**}\to X$ be such a contractive projection. Let $Q:\prod_\U X_i\to X^{**}$ be the contractive linear map defined in eq. (\ref{eq:Q}), then $\pi QJ: L\to X$ is a contractive linear map which clearly coincides with the identity map on $X$; this defines a contractive linear projection $P$ on $L$ with range $X$. If $U$ is a sign change in $L$ such that $UX$ is a vector sublattice of $L$, then $UP\overline U$ is a contractive projection from $L$ onto $UX$.
\end{proof}

Specific applications of the preceding Theorem \ref{L-inverse ultraprod} will be given in \S  \ref{subsec:applic-C-pav} below and in section \ref{script-C}. Now we specialize Theorem  \ref{L-inverse ultraprod} to the case where $X_i=X$ for every $X$ and get at once the following ultraroot result:

\begin{corollary}\label{L-ultraroot} Let $L$ be an order continuous Banach lattice satisfying DPIU. Assume that $X$ is a Banach space which has an ultrapower that is linearly isometric to $L$. Then $X$ itself is linearly isometric to a closed sublattice of $L$, which can be chosen contractively complemented if moreover $L$ is a KB-space.
\end{corollary}

However we can say more about closed sublattices of $L$ that have an ultrapower linearly isometric to $L$.

\begin{proposition}\label{L-ultraroot-2}
 Let $L$ be a  Banach lattice satisfying DPIU, and assume that $E$ is a closed vector sublattice of $L$ such that $E_\U$ is linearly isometric to $L$. Then $E_\U$ is lattice isometric to $L$. If $L$ is a KB-space then $E$ is lattice isometric to a positively and contractively complemented closed vector sublattice of  $L$.
\end{proposition}

\begin{proof}
 a) Let $J: L\to E_\U$ be a surjective linear isometry, $i:E\to L$ be the inclusion map  and $i_\U: E_\U\to L_\U$ be its ultrapower map. Then by property DPIU the linear isometry $i_\U J: L\to L_\U$ is disjointness preserving. Since $i_\U$ is a lattice embedding it results that $J$ is also disjointness preserving. By the structure theorems for disjointness preserving maps, Fact \ref{structure-dp-maps} or \ref{structure-comp-dp-maps}, its modulus $|J|$ exists and is a vector lattice isometry from $L$ onto $E_\U$.
 
 b) Consider the vector lattice isomorphism $V= |J|^{-1}D_E$, its range $F=V(E)$ is a closed vector sublattice of $L$, and let $V_\U: E_\U\to F_\U$ be the ultrapower map of $V: E\to F$. Then $K= V_\U |J|: L\to F_\U $ is a surjective lattice isometry wich extends $D_F: F\to F_\U$. Since $F$ is a KB-space it is positively contractively complemented in its bidual by some projection $\pi$. Let $Q$ be the weak* limit  map $F_\U\to F^{**}, [x_i]_\U\mapsto w^*\lim x_i$, then $\pi QK: L\to L$ is a positive contractive projection from $L$ onto $F$. 
 \end{proof}

\begin{remark}
 It follows from Proposition  \ref{L-ultraroot-2} that in Corollary \ref{L-ultraroot} a closed vector sublattice of $L$,  linearly isometric to $X$ may be found that is not only contractively, but also positively complemented. Moreover if the ultrapower of $X$ is relative to a countably incomplete ultrafilter, then by Fact \ref{fact:ultra-oc}, the additional hypothesis ``$L$ is a KB-space'' at the end of  the  statement \ref{L-ultraroot} may be omitted.  
\end{remark}

\subsection{Application to Banach axiomatizability of classes of Banach lattices}

\begin{theorem}\label{axiomatizability1}  Let $\C$  be an axiomatizable class of order continuous Banach lattices with property DPIU.  Then the class $\C^B$ of Banach spaces linearly isometric to members of $\C$ is axiomatizable.
\end{theorem}

\begin{proof}  Clearly $\C^B$ is closed under ultraproducts since $\C$ is. As for the ultraroot condition, assume that a Banach space $X$ has an ultrapower $X_\U$ which is linearly isometric to a member $L$ of  $\C$, then, by Corollary \ref{L-ultraroot},  $X$ itself is linearly isometric to a sublattice $E$ of $L$, and by Proposition \ref{L-ultraroot-2}, the lattice ultrapower $E_\U$ is lattice isometric to $L$. By axiomatizability of $\C$, this class contains the lattice $E$  and thus the space $X$ belongs to~$\C^B$.
\end{proof}

The following corollary states a well known fact (see e.g. Theorem 6.3 in the survey \cite{CWH}), that may be given now a new proof.

\begin{corollary} For $1\le p<\infty$ the class of $L_p$-spaces is axiomatizable (in the language of Banach spaces).
\end{corollary}

\begin{proof} The case $p=2$ being trivial we may assume $p\ne 2$. The class of $L_p$ Banach lattices is axiomatizable since it is closed under ultraproducts and sublattices. On the other hand for every $p\ne 2$, every linear isometry from a $L_p$-space into another one preserves disjointness, since in $L_p$ spaces the disjointness of two elements $x$ and $y$ is characterized by the equation
\begin{equation}
\|x+y\|^p+\|x-y\|^p=2(\|x\|^p+\|y\|^p)\tag{$EC_p$}
\end{equation}
(equality case in the Clarkson inequality), see e.g. \cite[Cor. 2.1]{La}. Thus the class of $L_p$-Banach lattices has property DPIU.
\end{proof}

Several new examples of classes satisfying the hypotheses of Theorem \ref{axiomatizability1} will be given in the forthcoming sections. In accordance with Remark \ref{rem:unif-ultra-oc} each of these classes  will consist in uniformly  $q$-concave Banach lattices. 

\subsection{Application to $\C$-pavable spaces}\label{subsec:applic-C-pav} \hfil

\smallskip
Recall that given a pair $(X,Y)$ of normed spaces (resp. vector lattices), and $\lambda\ge 1$, a linear (resp. vector lattice homomorphism) $T: X\to Y$ is called a $\lambda$-embedding if 
$$\forall x\in X,\ \lambda^{-1}\|x\|\le\|Tx\|\le \lambda\|x\|$$
$Z=TX$ will be said to be ``$\lambda$-isomorphic to $X$'' or a ``$\lambda$-isomorphic copy of $X$ in 
$Y$''.
 
\begin{definition}\label{def:pavable}
Let $\C$ be a class of Banach spaces (resp. Banach lattices). We say that a Banach space (resp. lattice) $X$ is $\C$-pavable if for every  finite subset $F$ of $X$ and every $\eps>0$ there exists a $(1+\eps)$-copy $Y$ in $X$ of some member $G$ of $\C$ 
such that every point of $F$ lies at a distance at most $\eps$ from $Y$.
\end{definition}

\begin{remark}\label{rem: paving}
Let $I=\mathcal P_f(X)\times (0,\infty)$, ($\mathcal P_f(X)$ is the set of finite subsets of $X$) and for each $i=(F,\eps)\in I$ let $G_i$ be a member of $\C$ and $T_i:G_i\to X$ a linear isomorphism (resp. a vector lattice isomorphism) as in Definition \ref{def:pavable}. Set $X_i=T_iG_i$ for all $i\in I$. Then $(X_i)$ is a paving of $X$ and one may find an adapted ultrafilter $\U$ such that $\|T_i\|\to 1$ and $\|T_i^{-1}\|\to 1$ with respect to $\U$: e.g., order $I$ by $(F,\eps)\le (G,\delta)$ iff $F\subset G$ and $\eps\ge\delta$ and take any ultrafilter containing all the sets $S_i:=\{j\in I, j\ge i\}$. The maps $T_i$ induce a map $\tilde T:\prod_\U G_i\to \prod_\U X_i$ which is a surjective linear isometry (resp. an isometric vector lattice isomorphism).
\end{remark}

\begin{remark}
In the Banach space setting there is an equivalent definition which sounds more customary:
{\it a Banach space $X$ is $\C$-pavable if for every  finite-dimensional linear subspace $E$ of $X$ and every $\eps>0$ there exists a member $G$ of $\C$ and an $(1+\eps)$-copy of $G$ in $X$ which contains $E$.}
\end{remark}
\begin{proof}
Indeed apply the definition \ref {def:pavable} with $F=\{x_1,\dots x_d\}$ an Auerbach basis of $E$, that is
\[ \sup_i|\lambda_i|\le \big\|\sum_i\lambda_i x_i\big\|\le \sum_i |\lambda_i|\]
for every system of scalars $\lambda_1,\dots,\lambda_d$.  We get a linear subspace $Y$ of $X$ which is $(1+\eps)$-linearly isomorphic to a member of $\C$ and $d$ points $y_1,\dots, y_d$ in $Y$ with $\|x_i-y_i\|\le \eps$, $i=1,\dots, d$. Then
\[\big\| \sum_i \lambda_i(y_i-x_i)\big\|\le \eps \sum_i |\lambda_i|\le d\eps\sup_i |\lambda_i|\le d\eps \big\| \sum_i \lambda_ix_i\big\|
\]
and
\[ (1-d\eps)\big\| \sum_i \lambda_ix_i\big\| \le  \big\| \sum_i \lambda_iy_i\big\|\le (1+d\eps)\big\| \sum_i \lambda_ix_i\big\|
\]
Let $E_1=\mathrm{span}\{y_1,\dots,y_d\}$, then by the Kadec-Snobar theorem there is a projection $P:Y\to E_1$ of norm less than $\sqrt d$. Let $V=\ker P$ be the associated supplement of $E_1$ in $Y$, then $E+V$
is a linear subspace of $X$ containing $E$. We exhibit a good isomorphism $S$ from $Y$ onto $E+V$ by setting for any $y=v+\sum_i\lambda_i y_i$ with  $v\in V$ 
\[ Sy=v+\sum_i\lambda_i x_i \]
Indeed
\[\|Sy-y\|\le \big\| \sum_i \lambda_i(y_i-x_i)\big\|\le d\eps\big\| \sum_i \lambda_ix_i\big\| \le {d\eps\over 1-d\eps}\big\| \sum_i \lambda_iy_i\big\| \le {d\sqrt d\,\eps\over 1-d\eps}\|y\|
\]
Then for $\delta={d\sqrt d\,\eps\over 1-d\eps}$ we have $(1-\delta)\|y\|\le \|Sy\|\le (1+\delta)\|y\|$.
\end{proof}

\begin{lemma}\label{embedding}
Let $\C$ be a class of Banach spaces (resp. Banach lattices). If a Banach space (resp. lattice) $X$ is $\C$-pavable then it embeds in an ultraproduct of members of $\C$. In the lattice setting, if moreover $X$ does not contain $c_0$, this embedded copy of $X$ is positively contractively complemented.
\end{lemma}

\begin{proof}
By Remark \ref{rem: paving} there is a paving $(X_i)_{i\in I}$ of $X$ by linear subspaces (resp. vector sublattices) and an adapted ultrafilter $\U$ such that $\prod_\U X_i$ is linearly isometric (resp. and lattice isomorphic) to an ultraproduct $\prod_\U G_i$ of members of $\C$. Then consider the canonical embeding $\Delta$ of $X$ into $\prod_\U X_i$ defined in subsection \ref{subsec:MainResult}.

In the lattice case, since the Banach lattice $X$ does not contains $c_0$, there is positive  contractive map 
$\pi: X^{**}\to X$. The weak*-limit map $Q:\prod_\U X_i\to X^{**}$ defined in eq. (\ref{eq:Q}) is also positive and contractive and 
$\pi Q\Delta=\pi i_X=\mathrm{id}_X$, hence $P=\Delta\pi Q$ is a (positive, contractive) projection from $\prod_\U X_i$ onto $\Delta X$.
\end{proof}

\begin{corollary}\label{cor:C=C^B}
a) If a class $\C$ of order continuous Banach lattices is closed under ultraproducts and positive contractive projections onto sublattices, then it equals the class of  $\C$-pavable Banach lattices. 

b) If moreover $\C$ consists of Banach lattices with property DPIU, then the class $\C^B$ (of Banach spaces which are linearly isometric to some member of $\mathcal C$) equals the class of $\C^B\!$
-pavable Banach spaces. 
\end{corollary}

\begin{proof}
a) By Fact \ref{fact:ultra-oc}, members of  $\C$ are KB-spaces.  
Then by Lemma \ref{embedding}, if $X$ is a $\C$-pavable Banach lattice  it embeds isometrically as Banach lattice into an ultraproduct $\prod_\U G_i$ of members of $\C$ and the embedded copy is positively and contractively complemented in $\prod_\U G_i$. By the hypothesis, $\prod_\U G_i$ and $X$ also belong to $\C$.
\smallskip

b) If $X$ is a $\C^B$-pavable Banach space, there is a paving of $X$ by a family of linear subspaces and an adapted ultrafilter $\U$ such that $\prod_\U X_i$ is linearly isometric to an ultraproduct of members of $\C$, that is to a member $L$ of $\C$. Since members of $\C$ do not contain $c_0$, it  results from Theorem \ref{L-inverse ultraprod} that $X$ is linearly isometric to a positively and contractively complemented sublattice of $L$, which by the hypothesis must belong to $\C$. Thus $X$ belongs to $\C^B$.
\end{proof}

\section{script-$\C$ classes}\label{script-C}

In our definition of $\C$-pavable Banach spaces (or lattices) there is no requirement on the dimension of the paving subspaces. If we require that the paving subspaces are finite dimensional, we get the notion of ``script-$\C$'' spaces.

\begin{definition}
Given a class $\C$ of Banach spaces (resp. lattices), we say that a Banach space (resp. Banach lattice) is a script-$\C$-space (-lattice), in short $\mathcal{SC}$-space (-lattice), if it is $\C_f$-pavable, where $\C_f$ is the subclass of finite dimensional members of $\C$.
\end{definition}

Note that the class of script-$\C$-Banach spaces (resp. lattices) coincides with the class of $\C$-pavable Banach spaces (lattices) iff every member of $\C$ is a script-$\C$-Banach space (lattice).

Let us make a short digression about axiomatizability of classes $\mathcal{SC}$, and deduce an informal ``axiomatization" of axiomatizable classes $\C$ such that $\C=\mathcal{SC}$. 

\begin{proposition}
\parindent=0pt
Let $\C$ be a class of Banach spaces or lattices.\smallskip

i) $\mathcal{SC}$ is closed under ultraroots.\smallskip

ii) $\mathcal{SC}$ is closed under ultraproducts iff for each natural number $n$ and any real number $\eps\in(0,1]$, there is a natural number $m(n,\eps)$ such every member $X$ of $\C_f$ verifies the following property:

\smallskip
$(P_{n,\eps,m})$: every subset $F$ of cardinality $n$  of the unit ball of $X$ lies in an $\eps$-neighborhood of an $(1+\eps)$-linear (resp. vector lattice) isomorphic copy in $X$ of a member $G$ of $\C_f$ of dimension  $\le m$.

\smallskip
Moreover in this case every member of $\mathcal{SC}$ verifies $(P_{n,\eps,m'})$, with $m'=m(n,\eps/9)$.
\end{proposition}

\begin{proof}
i) Let $X$ be an ultraroot of a member $X_\U$ of $\mathcal{SC}$. If $F\subset X$ is a finite subset of the unit ball, and $\eps>0$, the canonical image $F_\U$ of $F$ in $X_\U$ is contained in a $(1+\eps)$-linear (or vector lattice) isomorphic copy $E$ of a member of $\C_f$ in $X_\U$. Since $E$ is finite-dimensional, it has the form $\prod_\U E_i$, where the $E_i$ are linear subspaces (or vector sublattice) of $X$, and $\dim E_i=\dim E$. Similarly if $T: G\to E$ is an $(1+\eps)$-isomorphism, then $T=\prod_\U T_i$ where $T_i:G\to E_i$ are $(1+\eps_i)$-isomorphisms, and $\lim_{i,\U} \eps_i=\eps$. Finally, again since $E$ is finite dimensional, it is easy to see that for each $x\in F$, $d(x,E)=\lim_{i,\U} d(x,E_i)$. Consequently there is $U\in\U$ such that for every $i\in U$, $F$ lies in a $2\eps$-neighborhood of $E_i$, and $E_i$ a $(1+2\eps)$-copy of $G$.

ii) a) Assume first the existence of the bounds $m(n,\eps)$ as in the statement. We prove first that each member $X$ of $\mathcal{SC}$  verifies $P_{n,\eps,m'}$. Indeed if $\delta=\eps/9$ and $F=\{x_1,\dots x_n\}$ is a finite set in the unit ball of $X$ there is an $(1+\delta)$-isomorphic embedding $T$ from a member $G$ of $\C_f$ into $X$, with $TG$ containing a set $F'=\{x'_1,\dots x'_n\}$ such that  $\|x_i-x'_i\|\le \delta$, $j=1\dots,n$. Note that elements of $F'$ have norm bounded by $1+\delta$, so that elements of $T^{-1}F'$ have norm bounded by $(1+\delta)^2$. By property $P_{n\delta}$ for $G$, there is a further member $G_1$ of $\C_f$, of dimension $\le m(n,\delta)$ and an $(1+\delta)$-isomorphic embedding $T_1:G_1\to G$ such that $d((1+\delta)^{-2}T^{-1}x'_i, T_1G_1)\le \delta$, $i=1,\dots,n$. Then $TT_1G_1$ is an $(1+\delta)^2$-isomorphic copy of $G_1$ in $X$ (note that $(1+\delta)^2\le 1+3\delta$), and $d(x_i,TT_1G_1)\le\delta(1+(1+\delta)^3)\le 9\delta$, $i=1,\dots,n$.

Let now $X=\prod_\U X_i$ be an ultraproduct of a family $(X_i)$ of members of $\mathcal{SC}$ and $F$ be a subset ot the unit ball of $X$ of finite cardinality $n$. We may identify $F$ with $\prod_\U F_i$, where each $F_i$ is a subset ot the unit ball of $X_i$ and has cardinality $n$. For each $i$ choose a member $G_i$ of $\C_f$ of dimension $\le m(n,\eps/9)$ and an $(1+\eps)$-isomorphic embedding of $G_i$ in $X_i$, such that $F_i$ lies in an $\eps$-neighborhood of $T_iG_i$. Then $T=\prod_\U T_i$ is an $(1+\eps)$-isomorphic embedding of $G=\prod_\U G_i$ in $X$ and $F$ lies in an $\eps$-neighborhood of $TG$. The space $G$ is a Banach space (resp. lattice), of finite dimension $d=\lim_{i,\U}\dim G_i\le m(n,\eps/9)$, but is perhaps not a member of $\C_f$. However by \cite[prop 6.1]{H} there is $U\in \U$ such that for every $i\in\U$, $G_i$ is $(1+\eps)$ isomorphic to $G$ as normed linear space. The proof of \cite[prop 6.1]{H} can be mimiked in the lattice case, starting with a  basis of $G$ consisting of positive disjoint vectors, giving then that $G_i$ is $(1+\eps)$-isomorphic to $G$ as normed vector lattice. Finally $F$ is $(1+\eps)^2$-isomorphic to some $G_i$, a member of $\C_f$.

b) Assume now that for some $n,\eps$ there is no bound $m(n,\eps)$ as in the statement, and let us prove that $\mathcal{SC}$ is not closed under ultraproducts. Indeed for every $m\ge 1$ there would exist a member $X_m$ of $\C_f$,  a subset $F_m$ of the unit ball of $X_m$ having cardinality $\leq n$, such that  any  linear subspace (resp. vector sublattice) of $X_m$  which is $(1+\eps)$-isomorphic to a member of $\C_f$, and intersect all the balls $B(x,\eps)$, $x\in F_m$, must have dimension $> m$.

Let $\U$ be a free ultrafilter on $\mathbb N$ and consider the space $X=\prod_\U X_m$ and the set $F= \prod_\U F_m$. The set $F$ has $n$ elements $\xi_k=[x_{k,m}]_\U$, where for each $m$, $F_m=\{x_{1,m}, x_{2,m},\dots x_{n,m}\}$ (note that $\xi_k$ is well defined since each sequence $(x_{k,m})_m$ is bounded). Consider, if it exists, a subspace (or sublattice) $E$ of $X$, of finite dimension $d$, which is $(1+\eps/2)$-isomorphic to a member $G$ of $C_f$ and intersects the balls $(B(\xi_k,\eps/2)$. Then $E=\prod_\U E_m$, where $E_m$ is a $\le d$-dimensional subspace (resp. vector sublattice) of $X_m$. For  some $U\in\U$, and every $m\in U$, $E_m$ is $({1+\eps\over 1+\eps/2})$-isomorphic to $E$, and thus  $(1+\eps)$-isomorphic to $G$, and  moreover intersects $B(x_{k,m},\eps)$ for every $k=1,\dots n$. Thus $\dim E_m> m$,  for all $m\in U$, a contradiction. Hence such an $E$ does not exist, and $X$ is not in $\mathcal{SC}$.
\end{proof}

\begin{corollary}\label{axioms}
Let $\C$ be a class of Banach spaces or lattices.
The following assertions are equivalent:
\parindent= 10pt

i) $\C$ is axiomatizable and $\C=\mathcal{SC}$

ii) There exists a map $m: \nat\to \nat$ such that $\C$ consists of those Banach spaces (resp. lattices) verifying the list of ``axioms'':

\vskip 3pt
$(A_n)$: \hfill\parbox[t]{13.5cm}{for every subset $F$ of cardinality $\le n$ of the unit ball of $X$ there is  a member $G$ of $\C_f$ of dimension  $\le m(n)$ and a $(1+\frac 1n)$-isomorphic copy of $G$ in $X$ which intersects all the balls $B(x,1/n)$, $x\in F$.}
\end{corollary}

The content of the following theorem is that the situation described in Corollary \ref{axioms} ``passes from Banach lattices to Banach spaces" for  classes of order continuous Banach lattices with DPIU.

\begin{theorem}
Let $\C$ be a class of order continuous Banach lattices with property DPIU, which is closed under ultraproducts. If $\C$ coincides with the class of script-$\C$-Banach lattices, then $\C^B$ coincides with the class of script-$\C^B$-Banach spaces. 
\end{theorem}

\begin{proof}
It is clear that a paving of a Banach lattice $X$ by a family $(X_i)$ of members of $\C_f$ induces a paving of the underlying Banach structure $X^B$ by the family $(X^B_i)$ of members of $\C^B_f$, hence $\C^B$ is included in $\mathcal S(\C^B)$.

Conversely let $X$ be a $(\C^B)_f$-pavable Banach space, and consider a paving of $X$ by a family $(X_i)$ of finite dimensional linear subspaces and an adapted ultrafilter $\U$ so that $\prod_\U X_i$ is linearly isometric to an ultraproduct $L$ of members of $\C$, which is also a member of $\C$ by hypothesis. By Theorem \ref{L-inverse ultraprod} and its proof we may assume that $X$ is a closed vector sublattice of $L$, and that there exists a linear isometry $J$ from $L$ onto $\prod_\U X_i$ such that $J\circ j=\Delta$, where $j$ is the inclusion map from $X$ into $L$ and $\Delta$ is the canonical linear isometric embedding of $X$ into $\prod_\U X_i$. Let $\tilde \gamma=\prod_\U \gamma_i$ be the ultraproduct of the inclusions  $\gamma_i: X_i\to X$, then $\tilde \gamma$ is a linear isometric embedding from $\prod_\U X_i$ into $X_\U$, such that $\tilde \gamma\circ \Delta=D_X$ (the canonical embedding of $X$ into $X_\U$). Let also $\tilde\jmath: X_\U\to L_\U$ be the  ultrapower of the inclusion map $j: X\to L$. We set $\phi=\tilde \gamma\circ J$ which is a linear isometry. Finally we have the commutative diagram
\[\xymatrix{&&\!\!\prod_\U X_i\ar^{\tilde\gamma}[dd]&&\\X\ar^\Delta[urr]\ar_{\phantom{jjj}j}[r]\ar_{D_X}[drr]&L\ar_J[ur]\ar^\phi[dr]&&\\& &X_\U \ar^{\tilde\jmath}[r]&L_\U }\]

Since $L$ has property DPIU the  linear isometry  $\tilde j\phi:L\to L_\U$ is disjointness preserving, and since $\tilde j$ is a vector lattice isometry, the map $\phi$ itself must be also disjointness preserving. Since $L_\U$ does not contain $c_0$ (like any member of $\C$, due to Fact \ref{fact:ultra-oc}), then  $X_\U$ does not contain $c_0$ as well, and thus by the structure theorem for disjointness preserving isometries we have $\phi=U|\phi|$, where $U$ is a sign change on $X_\U$. For $x\in X_+$ we have $j(x)\ge 0$ and $\phi(j(x))=D_X(x)\ge 0$, hence $|\phi|(j(x))\ge 0$ and $U(|\phi|(j(x)))\ge 0$, which in turn implies that  $U(|\phi|(jx))=|\phi|(jx)$ (since $U$ is a sign change) and finally $\phi(jx)=|\phi|(jx)$. By linearity this remains true for every $x\in X$.

Now we show that the Banach lattice $X$ is $\C_f$-pavable. By the hypothesis it will show that the Banach lattice $X$ is a member of $\C$, and thus the underlying Banach space  $X$ is a member of $\C^B$. Consider a finite subset $F$ in $X$ and $\eps>0$. Since $L$ is a member of $\C$, it is $\C_f$-pavable by the hypothesis, hence there is $G\in \C_f$ and a vector sublattice $Y$ of $L$, and a vector lattice isomorphism $T: G\twoheadrightarrow Y$ with $(1+\eps)^{-1}\|x\|\le \|Tx\|\le (1+\eps)\|x\|$ for every $x\in G$ and $\mathrm{dist}(x,Y)\le \eps$ for every $x\in j(F)$. Let $Z=|\phi|(Y)$ and $S=|\phi|\circ T$, then $Z$ is a sublattice of $X_\U$, $S$ verifies the same estimates than $T$ and  $\mathrm{dist}(x,Z)\le \eps$ for every $x\in |\phi|\circ j(F)=D_X(F)$. Since $G$ is a finite dimensional vector lattice it is generated by a finite system of positive atoms $e_1,\dots e_n$. Then $u_j=Se_j, j=1,\dots,n$ is a finite set of positive atoms generating $Z$. We may find a system of bounded families $((u_{1,i}),\dots, (u_{n,i}))$ in $X$, with $[u_{j,i}]_\U=u_j$, $j=1,\dots,n$ and $\|u_{j,i}\|=\|u_j\|$, $u_{j,i}\ge 0$ and $u_{j,i}\wedge u_{k,i}=0$ for all $j\ne k$ and all $i\in I$. Clearly $Z_i=\mathrm{span}\,[u_{1,i},\dots,u_{n,i}]$ is a sublattice of $X$, and $S_i(\sum_{j=1}^n\lambda_j e_j)= \sum_{j=1}^n\lambda_j u_{j,i}$ defines a 
vector lattice isomorphism from $G$ onto $Z_i$. It is not hard to see that $\|S_i\|\to \|S\|$, $\|S_i^{-1}\|\to \|S^{-1}\|$ with respect to $\U$. For each $x\in F$ we can find an element $z_x\in Z$ with $\|D_X x-z_x\|<\eps$. We have $z_x=Sg_x$ for a certain $g_x\in G$. Then $\lim_{i,\U} \|x-S_ig_x\|=\|D_X x-z_x\|<\eps$. Finally the set
\[\{i\in I: \|S_i\|<1+2\eps,  \|S_i^{-1}\|<1+2\eps, \mathrm{dist}\,(x,Z_i)<\eps \hbox{ for every } x\in F\}\]
belongs to $\U$, and is thus not empty.
\end{proof}

\begin{remark}
A simple case of an axiomatizable class $\C$ of Banach lattices with $\C=\mathcal{SC}$ is when $\C$ is closed under ultraproducts and sublattices. Examples are $L_p$-spaces and  classes of convex Musielak-Orlicz spaces satisfying a given $\Delta_2$-condition. Other examples of axiomatizable classes $\C$ with $\C=\mathcal{SC}$ but not closed under sublattices are Nakano spaces with exponents in a given finite interval \cite{PR} and classes $BL_pL_q$ \cite{HR}. These examples have $DPIU$ property whenever they are exactly $r$-convex for some $r>2$ (see section \ref{r-convex-BL}) but in the case of $BL_pL_q$ classes  other cases may be found (section \ref{sec:other}).
\end{remark}

\section{The case of  $r$-convex  Banach lattices, $r>2$}\label{r-convex-BL}

\subsection{Norm determination of  disjointness in $r$-convex  Banach lattices, $r>2$}\hfill\break

Some useful information about $r$-convexity and concavity in abstract Banach lattices and related matters used in the present section may be found in \cite[sec. 1.d]{LT}.

\begin{definition} We say that a Banach lattice $X$ is {\it exactly $r$-convex} if it is $r$-convex with $r$-convexity constant equal to one, i. e.
\[\forall  x,y\in X\quad  \|(|x|^r+|y|^r)^{1/r}\|\le (\|x\|^r+\|y\|^r)^{1/r}\]
Similarly, $X$ is exactly $s$-concave if
\[\forall  x,y\in X\quad  \|(|x|^s+|y|^s)^{1/s}\|\ge (\|x\|^s+\|y\|^s)^{1/s}\]
\end{definition}

\begin{proposition}\label{disjointcharacter} 
Let $r>2$ and $X$ be a strictly monotone and exactly $r$-convex Banach lattice.
Then two elements $x,y\in X$ are disjoint if and only if they satisfy the condition 
\[
[\mathrm D_r]:\ \ \  \forall t\in [0,1]\ \ \|x+ty\|^r+\|x-ty\|^r\le 2 (\|x\|^r+t^r\|y\|^r)
\]
\end{proposition}

\begin{proof} Clearly condition $[\mathrm D_r]$ is necessary since if $x\perp y$ then for every positive real number 
$t$ we have
$$\|x\pm ty\|^r=\|(|x|^r+t^r|y|^r)^{1/r}\|^r\le \|x\|^r+t^r\|y\|^r$$
Let us prove now that condition $[\mathrm D_r]$ is sufficient.

Recall that $X$ being exactly $r$-convex is also exactly $2$-convex (since $r\ge 2$) \cite[Prop.~1.d.5]{LT}. We have then for every $x,y\in X$:
\begin{align*}
\left({\|x+y\|^r+\|x-y\|^r\over 2}\right)^{1/r}&\ge \left({\|x+y\|^2+\|x-y\|^2\over 2}\right)^{1/2}  
\ge \left\|\left({|x+y|^2+|x-y|^2\over 2}\right)^{1/2}\right\|  \\ &= \left\|(|x|^2+|y|^2)^{1/2}\right\|
\end{align*}
Then for every $x,y\in X$ and $t$ a positive real number:
\[
\|x+ty\|^r+\|x-ty\|^r\ge 2\left\|(|x|^2+t^2|y|^2)^{1/2}\right\|^r \ge 2\left\|(|x|^2+t^2(|x|\wedge |y|)^2)^{1/2} 
\right\|^r
\]
Consider the 2-concavification $X^{(2)}$ of the Banach lattice $X$ and denote by $x\mapsto x^2$ the natural (non-linear) bijective map from $X$ onto $X^{(2)}$. We have
$$\left\|(|x|^2+t^2(|x|\wedge |y|)^2)^{1/2}\right\|_X=\left\||x|^2+t^2(|x|^2\wedge |y|)^2)\right\|^{1/2}_{X^{(2)}}$$
Recall that $X$ being 2-convex,  its 2-concavification $X^{(2)}$ is a Banach lattice too. Let $\ph\in  X^{(2)*}$ be a norm one positive functional norming the element $|x|^2$. We have
\begin{align*}
\left\||x|^2+t^2(|x|^2\wedge |y|)^2)\right\|_{X^{(2)}} &\ge \langle\ph, |x|^2+t^2(|x|^2\wedge |y|^2)\rangle = \langle \ph, |x|^2\rangle + t^2\langle \ph, |x|^2\wedge |y|^2\rangle \\
&= \|x\|_X^2+ t^2\langle \ph, |x|^2\wedge |y|^2\rangle
\end{align*}
it follows that
\[
\left\|(|x|^2+t^2(|x|\wedge |y|)^2)^{1/2}\right\|_X^r\ge \left(\|x\|_X^2+ t^2\langle \ph, |x|^2\wedge |y|^2\rangle\right)^{r/2}\ge \|x\|^r+\frac r2 t^2\|x\|^{r-2} \langle \ph, |x|^2\wedge |y|^2\rangle
\]
where we used the elementary inequality $(a+b)^\alpha\ge a^\alpha+\alpha ba^{\alpha-1}$, valid for positive real numbers $a,b$ and exponent $\alpha\ge 1$. Hence
$$\|x+ty\|^r+\|x-ty\|^r\ge 2\|x\|^r+r t^2\|x\|^{r-2} \langle \ph, |x|^2\wedge |y|^2\rangle $$
Assume now that $x,y\in X$ verify condition (D). Then necessarily
$$r t^2\|x\|^{r-2} \langle \ph, |x|^2\wedge |y|^2\rangle\le 2 t^r\|y\|^r$$
for every $0<t\le 1$. Dividing by  $t^2$ and letting $t\to 0$, we deduce since $r>2$ 
$$ \langle \ph, |x|^2\wedge |y|^2\rangle=0$$
Then
$$\|x^2\|_{X^{(2)}}=\langle \ph, |x|^2\rangle= \langle \ph, |x|^2-|x|^2\wedge |y|^2\rangle 
\le \left\||x|^2-|x|^2\wedge |y|^2 \right\|_{X^{(2)}} $$
that is
$$\|x\|^2_X\le  \|(|x|^2-|x|^2\wedge |y|^2)^{1/2} \|_X$$
Since $0\le (|x|^2-|x|^2\wedge |y|^2)^{1/2} \le |x|$ and $X$ has strictly monotone norm the last inequality implies
$|x|=(|x|^2-|x|^2\wedge |y|^2)^{1/2}$, that is $|x|\wedge |y|=0$.
\end{proof}

\begin{corollary}\label{DPisometries} Let $X$,  $Y$ be  exactly $r$-convex Banach lattices ($r>2$), and assume that $Y$ has a  strictly monotone norm. Then every linear isometry $T: X\to Y$ preserves disjointness. 
\end{corollary}

\subsection{Some consequences}

\begin{corollary}\label{axiomatizability2} Let $2<r\le s<\infty$ or $1<r\le s<2$. The class $\mathcal L^B_{r,s}$ of Banach spaces which are linearly isometric to exactly $r$-convex and $s$-concave Banach lattices is axiomatizable.
\end{corollary}

\begin{proof} The class $\mathcal L_{r,s}$ of exactly $r$-convex and $s$-concave Banach lattices is axiomatizable since is trivially closed under ultraproducts and sublattices. Being $s$-concave, these spaces do not contain $c_0$ as sublattices, and thus are order continuous (\cite[Thm 1.c.4 and the Remark that follows]{LT}). Since this  $s$-concavity is exact, these spaces are also strictly monotone. When $2<r\le s<\infty$ it results from   Corollary \ref {DPisometries} that every member of $\mathcal L_{r,s}$ has property DPIU.  By Theorem \ref{L-ultraroot} the class $\mathcal L^B_{r,s}$  is axiomatizable. This last conclusion remains true if $1<r\le s<2$ since $\mathcal L_{r,s}$ is the class consisting of conjugate spaces to members of $\mathcal L_{s',r'}$, where $r',s'$ are the conjugate exponents of $r,s$ (thus $2<s'\le r'<\infty$) and these classes consist of superreflexive spaces (for which conjugation commutes with ultrapower functors).
\end{proof}

More generally we have:

\begin{corollary} Let $2<r\le s<\infty$ or $1<r\le s<2$. Let $\mathcal L$  be an axiomatizable class of Banach lattices included in the class of exactly $r$-convex and $s$-concave Banach lattices. Then the class $\mathcal L^B$ of Banach spaces linearly isometric to members of $\mathcal L$ is axiomatizable.
\end{corollary} 

\begin{corollary} In particular for $2<p,q<\infty$ or $1<p,q<2$ the class $\mathcal BL_pL_q^{B}$ of Banach spaces linearly isometric to bands in spaces $L_p(L_q)$,  and for $2<r\le s<\infty$ or $1<r\le s<2$  the class  $\mathcal N_{r,s}^B$ of Nakano (Banach) spaces associated with a variable exponent $p(\cdot)$ with values in $[r,s]$ are axiomatizable.
\end{corollary}

\begin{proof} Recall that the class $\mathcal BL_pL_q$ of  Banach lattices isometrically isomorphic to bands in spaces $L_p(L_q)$ is axiomatizable, see \cite{HR}, and for $1\le r\le s<\infty$ the class $\mathcal N_{r,s}$ of Nakano Banach lattices  with a variable exponent $p(\cdot)$ taking values in $[r,s]$ is axiomatizable, a result of Poitevin in \cite{P} (see also \cite{PR}). The corollary follows for $\mathcal BL_pL_q^B$ in the case $2<p,q<\infty$, resp. $\mathcal N_{r,s}^B$ in the case $2<r\le s<\infty$. The other cases are proved by duality: indeed these classes consist of reflexive Banach spaces, and it is well known (see e.g. \cite{H}) that if an ultraproduct $\prod_\U X_i$ is reflexive, then $(\prod_\U X_i)^*=\prod_\U X_i^*$ (up to a canonical isomorphism). It is then easy to see that for $1<p,q<\infty$, the axiomatizability of  $\mathcal BL_pL_q^B$ is equivalent to that of $\mathcal BL_{p'}L_{q'}^B$, where $p',q'$ are the conjugate exponents to $p,q$. The case of  Nakano spaces is a little more involved, since one has to take care of the fact that the dual norm of the Luxemburg norm is an Orlicz norm, that is only equivalent within a factor 2 to a Luxemburg norm. Thus for $1<r\le s<2$ the result for Nakano spaces and Luxemburg norm has to be derived by duality from the case $2<r\le s<\infty$ for Orlicz norm. Note however that Poitevin's result is stated in Luxemburg norm, but it holds true also by duality for the Orlicz norm in the whole range $1<r\le s<\infty$.
\end{proof}

Between the two preceding classes of examples are classes of Orlicz lattices associated with a disjointly additive convex modular satisfying a $\Delta_2$ condition \cite{CZ}. These Banach lattices are representable as Musielak-Orlicz spaces  associated with a convex Musielak-Orlicz function satisfying an uniform $\Delta_2$ estimation \cite{W81,W84}, and vice-versa. The class of convex Orlicz lattices satisfying a given $\Delta_2$ condition is closed under ultraproducts \cite{W84} and sublattices, thus axiomatizable (as Banach lattices), in both Luxemburg and Orlicz norms.

\begin{corollary}
For $2<r\le s<\infty$ or $1<r\le s<2$ the class $\mathcal {OL}_{r,s}^B$ of Banach spaces linearly isometric to an Orlicz lattice with exactly $r$-convex, $s$-concave modular is axiomatizable.
\end{corollary}

Let us finish this section with mentioning an open question concerning classes of Nakano spaces.

\begin{question}
 Is the class $\mathcal N_{r,s}^B$ of Nakano (Banach) spaces associated with any variable exponent $p(\cdot)$ with values in $[r,s]$ axiomatizable, for every $1\le r< s<\infty$? 
\end{question}

\section{Other classes with DPIU}\label{sec:other}

In this section we investigate the DPIU property for the classes of $\mathcal BL_pL_q$-Banach lattices, showing that it is satisfied  for a far larger set of indices $(p,q)$ than simply $(2,+\infty)\times (2,+\infty)$ (the case which is settled by Corollary \ref{DPisometries}). We shall use well known informations about the (linear) isometric embeddings of $L_p(L_q)$ spaces. Before, we give a simple lemma that will allow us to pass later from the real case to the complex case.

\begin{lemma}\label{re-dp-implies-comp-dp}
Let $L$ be a real Banach lattice and $1\le p,q<\infty$. If every linear isometry from $L$ into any $L_p(L_q)$ space is disjointness preserving, then any complex linear isometry from $L_\mathbb C$ into $(L_p(L_q))_\mathbb C$ is also disjointness preserving.
\end{lemma}

\begin{proof}
Let $T: L_\mathbb C \to (L_p(L_q))_\mathbb C$ be a complex linear isometry. The restriction $T_r$ of $T$ to $\Re L_\mathbb C = L$ is a real linear isometry. As a real linear space, $(L_p(L_q))_\mathbb C$ is linearly isometric to $L_p(L_q)(\ell_2^2)$. Let $G_1,G_2$ be two independent Gaussian variables on $[0,1]$, that are normalized in $L_q[0,1]$. Then the map $S: z\mapsto (\Re z) \,G_1+ (\Im z)\, G_2$ is a real linear isometry from $\mathbb C$ into $L_q[0,1]$, and thus
\[ f\mapsto \widetilde Sf:=\Re f\otimes G_1+\Im f\otimes G_2\] 
defines a real linear isometry $\widetilde S$ from $(L_p(L_q))_\mathbb C$ into $L_p(L_q)(L_q[0,1])=L_p(L_q(L_q[0,1]))$. Note that for every $z\ne 0$, the function $Sz$ has full support in $[0,1]$. In particular $Sz_1$ and $Sz_2$ may be disjoint only if $z_1=0$ or $z_2=0$. It follows that $\widetilde Sf$, $\widetilde Sg$ are disjoint in the real Banach lattice $L_p(L_q)(L_q[0,1])$ iff $f,g$ are disjoint in the complex Banach lattice $(L_p(L_q))_\mathbb C$. Since by hypothesis the real linear isometry $\widetilde ST_r: L\to L_p(L_q(L_q[0,1]))$ necessarily preserves the disjointness, so does the map $T_r$. Now if $h_1,h_2\in L_\mathbb C$ are disjoint then both $\Re h_1$ and $\Im h_1$ are disjoint from both $\Re h_2$ and $\Im h_2$, thus $Th_1=T_r\Re h_1+iT_r \Im h_1$ is disjoint from $Th_2=T_r\Re h_2+iT_r \Im h_2$.
\end{proof}

\begin{notation}
Given a space $X=L_p(S,\Sigma,\sigma;L_q(S',\Sigma',\sigma'))$, denote by $N$ the map $X\to L_p(\Omega,\A,\mu)$ such that $N(f)(\omega)=\|f(\omega,\cdot)\|_q$ (the norm in $L_q$ of the partial function $f(\omega,\cdot)$). This map $N$ is called the {\it random norm} of the $L_p(L_q)$ space $X$ (\cite{LR,HR}). 

Clearly for all $x\in X$ and $\ph\in L_\infty(S,\Sigma,\sigma)$, we have $\ph.x\in X$ and $N(\ph x)=|\ph|N(x)$.  This is typically used with $\ph={N(v)\over N(w)}$, with $v,w\in X$ and $N(v)\le N(w)$ 
\end{notation}

\begin {proposition}\label{DP-for-LpLq-isometries} Let $1\le p\ne q<\infty$ with $q\ne 2$ and $q\not\in [p,2]$ in the case  $p\le 2$. Assume that the space  $L_q(\Omega',\A',\mu')$ has dimension greater or equal to two. If $p=1$ we assume that $L_q(\Omega')$ has dimension greater or equal to three. Then every linear isometry of $L_p(\Omega,\A,\mu;L_q(\Omega',\A',\mu')))$ into  another space $L_p(S,\Sigma,\sigma;L_q(S',\Sigma',\sigma'))$ is disjointness preserving.
\end{proposition}

\begin{proof} By Lemma \ref{re-dp-implies-comp-dp} we need only treat the real $L_p(L_q)$ case.
By  \cite[Theorem 5.1  and Proposition 1.6]{GR}, under the conditions of the Proposition and if $L_p(\Omega,\A,\mu)$ has dimension at least 2, then for every linear isometry of  $X:=L_p(\Omega,\A,\mu;L_q(\Omega',\A',\mu')))$ into itself there exists a positive linear isometry $\widetilde T$ from  $L_p(\Omega,\A,\mu)$ into itself such that 
$$N(Tf)=\widetilde T N(f)$$
 The proof is based on an analysis of isometric copies of 2-dimensional $\ell_p^2$ and $\ell_q^2$ spaces in $X$ (and sometimes 3-dimensional subspaces, namely $\ell_q^3$-subspaces if $p=1$, or $\reel\oplus_p(\reel\oplus_q\reel)$ in the cases $q<p<2$, or $p=2$ or $q=1$). The conditions on $X$ insure that these copies exist in $X$. Thus the theorem 5.1  in \cite{GR} extends trivally to isometries from $X$ into any $L_p(L_q)$ space $Y$,  with the same proof ($\widetilde T$ is now a positive linear isometry from $L_p(\Omega,\A,\mu)$ into $L_p(S,\Sigma,\sigma)$).
Since isometries between $L_p$-spaces are disjointness preserving, positive isometries are lattice homomorphisms. Thus the isometry $\widetilde T$ preserves the ``$q$-sums'':
$$\widetilde T((\ph^q+\psi^q)^{1/q})=((\widetilde T\ph)^q+(\widetilde T\psi)^q)^{1/q}$$
Disjointness in $X$ is characterized by the equation
$$N(f+g)^q+N(f-g)^q=2(N(f)^q+N(g)^q)$$
which implies by the preceding
$$N(Tf+Tg)^q+N(Tf-Tg)^q=2(N(Tf)^q+N(Tg)^q)$$
and the disjointness of the images $Tf,Tg$.

The case where $L_p(\Omega,\A,\mu)$ is trivial (dimension 1) is not formally evoked in \cite{GR}. In this case clearly the space $X$ cannot contain the three dimensional spaces $\reel\oplus_p(\reel\oplus_q\reel)$. However these subspaces were considered there only for proving that $T$ preserves the disjointness of $L_p$-supports, a property which is now trivial. The rest of the argument is valid and shows that there exists $\ph_0\in L_p(S,\Sigma,\sigma)$ such that $N(Tf)=\|f\| \ph_0$ for all $f\in X$ and thus again $T$ is disjointness preserving (this is also a direct consequence of \cite[prop. 5.4 and 5.8]{GR}).
\end{proof}

\begin{recall}
Let $1\le p\ne q<\infty$. A $\mathcal BL_pL_q$ Banach lattice $X$ is called {\it fiber-atomless} if it can be represented as a band in a space $L_p(L_q)$, where the Banach lattice $L_q$ is atomless (equivalently it is the $L_q$-space of an atomless measure space). 
Let $X$ be a band in a space $L_p( L_q)$ and $B$ the smallest band $B$ in $L_q$ such that $X\subset L_p(B)$.
Then $X$ is  fiber-atomless if the Banach lattice $B$ is atomless \cite[Lemma 2.3]{HR-11}. For given $p\ne q$, the class of fiber-atomless $\mathcal BL_pL_q$ Banach lattices is axiomatizable \cite[Proposition 2.5]{HR-11}.\end{recall}

\begin{lemma}\label{sep-fib-atl}
Let $X$ be a fiber-atomless $\mathcal \mathcal BL_pL_q$ Banach lattice, and $E$ be a separable subset of $X$. There is a separable closed vector sublattice $X_0$ of $X$, containing $E$ and which is itself a fiber-atomless $\mathcal BL_pL_q$ Banach lattice.
\end{lemma}

\begin{proof}
This follows immediately from Downwards L\"owenheim-Skolem theorem \cite[prop. 9.13]{HI} and the axiomatizabilty of the class of fiber-atomless $\mathcal BL_pL_q$ Banach lattices. Let us however give a direct (but tedious) proof  for the reader unfamiliar with this kind of logic arguments. Assume that $X$ is a band in $L_p(\Omega,\A,\mu;L_q(S,\Sigma,\nu))$. We may assume that $E$ is finite or countable, and that each element of $f\in E$ belongs to the algebraic tensor product $L_p(\Omega,\A,\mu)\otimes L_q(S,\Sigma,\nu)$, i. e. has the form $f=\sum_{i=1}^n g_i\otimes h_i$, $g_i\in L_p(\Omega,\A,\mu)$, $h_i\in 
L_q(S,\Sigma,\nu)$, $i=1\dots n$. Recall that 
\[ g\otimes h(\omega,s)= g(\omega)h(s) \]
Choose such a decomposition for each $f\in E$,  let $G_f$, resp. $H_f$ be the set of the first factors $g_i$, resp. $h_i$ appearing in this decomposition, and $G$, resp. $H$ be the closed vector sublattice generated by $\cup_{f\in E}G_f$, resp. $\cup_{f\in E}H_f$ in $L_p(\Omega,\A,\mu)$, resp. $L_q(S,\Sigma,\nu)$. Then $G$, $H$ are separable,  and isomorphic as vector lattices to some spaces $L_p(\Omega_0,\A_0,\mu_0)$, resp. $L_q(S_1,\Sigma_1,\nu_1)$. Note that $H$ is not necessarily diffuse, but it has at most countably many atoms $a_i$. Each of these atoms generates a band $B_i$ in $L_q(S,\Sigma,\nu)$ which is atomless. A standard measure-theoretic argument shows that $B_i$ contains an atomless separable closed sublattice $L_i$ containing $a_i$ as element. Replacing each component $\mathbb K. a_i$ in $H$ by $L_i$, we enlarge $H$ to a closed, separable and  atomless vector sublattice $H_0$ of $L_q(S,\Sigma,\nu)$, which is thus isomorphic to some separable atomless $L_q(S_0,\Sigma_0,\nu_0)$. Let $T:L_p(\Omega_0,\A_0,\mu_0)\to G$ and $U:L_q(S_0,\Sigma_0,\nu_0)\to H_0$ be two surjective vector lattice isomorphisms, then the linear map $T\otimes U: L_p(\Omega_0,\A_0,\mu_0)\otimes L_q(S_0,\Sigma_0,\nu_0)\to L_p(\Omega,\A,\mu)\otimes L_q(S,\Sigma,\nu)$ is well defined and verifies the equality  $N(T\otimes U(f))=TN(I\otimes U f)=T N(f)$ for every $f$ in its domain ($N: L_p(L_q)\to L_p $ is the random norm). Thus $T\otimes U$ is isometric for the $L_p(L_q)$ norms and extends by density to $L_p(\Omega_0,\A_0,\mu_0;L_q(S_0,\Sigma_0,\nu_0))$. This map is also positive and disjointness preserving, hence a lattice isometry. Its range $Y$ is a separable closed sublattice of $L_p(\Omega,\A,\mu;L_q(S,\Sigma,\nu))$, which contains $E$ by construction. Then $X_0=Y\cap X$ is an order-ideal in $Y$, and is order isometric to an order ideal (equivalently, a band) in $L_p(\Omega_0,\A_0,\mu_0;L_q(S_0,\Sigma_0,\nu_0))$.
\end{proof}

\begin{proposition}\label{fiber-atomless-DPIU}
Let $1\le p\ne q<\infty$ with $q\ne 2$ and $q\not\in [p,2]$ in the case  $p\le 2$. Then every linear isometry from a fiber-atomless $\mathcal BL_pL_q$ Banach lattice into any $\mathcal BL_pL_q$ Banach lattice is disjointness preserving. In particular every fiber-atomless $\mathcal BL_pL_q$ Banach lattice has property DPIU.
\end{proposition}

\begin{proof}
Let $X$ be a fiber-atomless $\mathcal BL_pL_q$ Banach lattice, and $T: X\to L_p(L_q)$ be a linear isometry (it is enough to consider this case where the target $\mathcal BL_pL_q$ Banach lattice is $L_p(L_q)$, since any $\mathcal BL_pL_q$ Banach lattice is embedded in some $L_p(L_q)$ Banach lattice). Let $x,y\in X$ be a pair of disjoint elements. By Lemma \ref{sep-fib-atl} there is a separable sublattice $Y$ of $X$, containing $x,y$ and which is itself a fiber-atomless $\mathcal BL_pL_q$ Banach lattice.  By the isometric classification of separable $\mathcal BL_pL_q$ Banach lattices \cite[p. 207]{HR}, $Y$ is  isometrically isomorphic as (real or complex) Banach lattice with a space $L_p(L_q[0,1])$, where $L_p$ is some separable $L_p$ space. An application of Proposition \ref{DP-for-LpLq-isometries} to $T\mid_Y$ shows now that $Tx, Ty$ are disjoint.
\end{proof}

\begin{corollary}
Let $1< p\ne q<\infty$ with $q\ne 2$ or $p=1$, $q\in(2,\infty)$. Then the class of fiber-atomless $\mathcal BL_pL_q$ Banach spaces is axiomatizable. 
\end{corollary} 

\begin{proof}
If $1\le p\ne q<\infty$ with $q\ne 2$ and $q\not\in [p,2]$ in the case  $p\le 2$, this results by Theorem \ref{axiomatizability1} from Proposition \ref{fiber-atomless-DPIU} and the fact that the class of fiber-atomless $\mathcal BL_pL_q$ Banach lattices is axiomatizable. In the remaining case $1<p<q<2$ this results by duality from the case $2<q<p<\infty$. 
\end{proof}

\begin{remark}
It is natural to exclude the case $q=2$ in the preceding statement since atomless Hilbert lattices cannot be axiomatized in Banach space language ($L_2[0,1]$ and $\ell_2$ are linearly isometric). By contrast the case $p=1$, $q\in (1,2)$ makes sense but it is not accessible by the techniques developed in this paper.
\end{remark}

Now we shall consider some axiomatizable classes of $BL_pL_q$-Banach lattices which are larger than the class of fiber-atomless $BL_pL_q$-Banach lattices, but still have DPIU property.

\begin{recall}[{\cite[section 2]{HR-11}}] Two elements $x$, $y$ of a space $L_p(L_q)$ are said {\it base-disjoint} if $N(x)\wedge N(y)=0$ where $N:L_p(L_q)\to L_q$ is the random norm (this notion makes sense only if $p\ne q$). An equivalent condition is that  $|x|\wedge |y|=0$ and $\|x+y\|^p=\|x\|^p+\|y\|^p$. This last condition may be taken as a definition for  base-disjoint elements in a $\mathcal BL_pL_q$-lattice $X$ and it is clearly invariant under Banach lattice isomorphisms. In particular the initial condition $N(x)\wedge N(y)=0$ will be realized or not simultaneously for all representations of $X$ as a band in a $L_p(L_q)$ space.
\end{recall}

\begin{definition}
Let us say that a band $B$ in a $\mathcal BL_pL_q$-Banach lattice is a base-band if every element of $X$ disjoint from $B$ is in fact base-disjoint from $B$.
\end{definition}

It is easy to see that base-bands in a $L_p(L_q)$-Banach lattice are the bands of the form $B(L_q)$, where $B$ is a band in $L_p$. More generally for any representation of a $\mathcal BL_pL_q$-Banach lattice X in a $L_p(L_q)$-Banach lattice, the base-bands of $X$ are exactly the traces $X\cap Z$ of the base-bands $Z$ in $L_p(L_q)$.

The base-band generated by an element $x$ of a $\mathcal BL_pL_q$-Banach lattice $X$ is the smallest base-band in $X$ containing $x$. We shall call it the {\it base-support} of $x$. If $X$ is represented as a band in some $L_p(L_q)$ space, and $N: L_p(L_q)\to L_p$ is the random norm, then the base-support of $x$ is $B(L_q)\cap X$, where $B$ is the band generated by  $N(x)$ in $L_p$. Thus $x,y\in X$ have the same base-support iff $N(x),N(y)$ generate the same band in $L_p$.

\begin{notation}
For $n\ge 1$,  we denote  by $\mathcal BL_pL_q^{\ge n}$ the subclass of $\mathcal BL_pL_q$ consisting of Banach lattices in which every base-band contains a lattice isometric copy of the $n$-dimensional lattice  $\ell_q^n$.
\end{notation}

\begin{lemma}\label{fiber-dlmension} Let $n\ge 1$ be an integer.

a) The Banach lattices $L_p(\ell_q^n)$ belong to $\mathcal BL_pL_q^{\ge n}$ but not to $\mathcal BL_pL_q^{\ge n+1}$.

b) If a $\mathcal BL_pL_q$-Banach lattice belongs to $\mathcal BL_pL_q^{\ge n}$ but not to $\mathcal BL_pL_q^{\ge n+1}$ it has a base-band lattice-isometric to a $L_p(\ell_q^n)$-lattice. 

c) In a Banach lattice belonging to $\mathcal BL_pL_q^{\ge n}$, for any nonzero element $x$ there is a system $(y_1,\dots,y_n)$ consisting of $n$ disjoint vectors with the same base-support as $x$ and  spanning  a $\ell_q^n$-sublattice.\end{lemma}

\begin{proof}
The key argument is that disjoint norm one elements $x_1,\dots x_k$ in a $\mathcal BL_pL_q$-Banach lattice $X$ (that we represent as a band in a $L_p(L_q)$-space with $L_p$-valued random norm $N$) generate a  $\ell_q^k$-subspace iff they have equal random norms: $N(x_1)=N(x_2)=\dots =N(x_k)$. 

Indeed the equality $\|x_1+x_2\|^q=\|x_1\|^q+\|x_2\|^q$ for disjoint vectors $x_1,x_2\in X$ means exactly that 
\[ \|N(x_1)^q+N(x_2)^q\|_{p/q}= \|N(x_1)^q\|_{L_{p/q}}+\|N(x_2)^q\|_{L_{p/q}}\]
which happens if and only if $N(x_1)^q$ and $N(x_2)^q$ are proportional (equality case in Minkowski inequality, if $p\ne q$ ); moreover  the coefficient of proportionality must be 1 since $\|x_1\|=\|x_2\|$.

a) It is clear that $L_p(\ell_q^n)$ belongs to $\mathcal BL_pL_q^{\ge n}$. Conversely given any system   $(x_1,\dots x_n)$ of $n$ disjoint  elements in $L_p(\ell_q^n)$ with the same random norm $\varphi$, any  element $f$ in $L_p(\ell_q^n)$ which is disjoint from $(x_1,\dots, x_n)$ must vanish (when viewed as a $\ell_q^n$-valued measurable function) on the support of $\varphi$, that is $N(f)$ is disjoint from $\ph$ (in $L_p$). Therefore we cannot complete the system $(x_1,\dots x_n)$ to a disjoint system spanning  a $\ell_q^{n+1}$-subspace.

b) First we remark that if $(y_1,\dots,y_n)$ is a system  of pairwise disjoint elements of a $\mathcal BL_pL_q$-lattice $X$  having the same base-support then we can find another  system  $(z_1,\dots,z_n)$ of disjoint vectors which generate a $\ell_q^n$-subspace. Indeed, set $z_j=\ph_j y_j$ with $\ph_j= (N(y_1)\wedge \dots\wedge N(y_n))/N(y_j)$.

Let $X$ be a $\mathcal BL_pL_q^{\ge n}$-lattice which  does not belong  to $\mathcal BL_pL_q^{\ge n+1}$. Let $Y$ be a base-band of $X$ which contains no $\ell_q^{n+1}$-sublattice. By hypothesis $Y$ contains a $\ell_q$-basis  $(x_1,\dots, x_n)$ consisting of positive, disjoint vectors. By the observation  above the base-band $Z$ generated by the system   $(x_1,\dots, x_n)$ contains no non zero element disjoint from all the $x_j$, $j=1\dots n$, otherwise it would contain an isometric copy of $\ell_q^{n+1}$ as subattice.Thus $Z$ coincides with the band generated by $(x_1,\dots, x_n)$. For the same reason  we cannot decompose any of the $x_j$'s in two disjoint, but not base-disjoint parts $x'_j,x''_j$ (that is the $x'_j$ are "fiber-atoms" in the sense of [HR]). It follows that $Z$ is lattice-isomorphic to  $L_p(A;\ell_q^n)$ where $A$ is the common support of the $N(x_j)$. This isomorphism from  $L_p(A;\ell_q^n)\approx L_p(A)^n$ onto $Z$ is simply given by $(f_1,\dots,f_n)\mapsto \sum_{j=1}^n {f_j \over N(x_j)}x_j$.

c) We show that for every subset $A$ of the support of $N(x)$ with positive measure we may find a system $(y_1,\dots,y_n)$ of disjoint vectors with $N(y_1)=\dots =N(y_n)=\chi_B N(x)$ where $B$ is a subset of  $A$ with non-zero measure. 

Indeed by hypothesis we may find disjoint norm one elements $x_1,\dots x_n$ in the base-band of $\chi_Ax$ which generate a $\ell_q^n$ subspace. Then $N(x_1)=\dots =N(x_n)=:\phi$ with $\mathrm{supp}\,\phi\subset \mathrm{supp}\,N(x)$. Let $B:= \mathrm{supp}\,\phi$ and $\psi:= {\chi_B \over \phi}N(x)$ (with the convention $0/0=0$) then put $y_j=\psi x_j$.

Now by a standard exhaustion argument we may find a system $(z_1,\dots,z_n)$ of disjoint vectors with $N(z_1)=\dots =N(z_n)=N(x)$.
\end{proof}

\begin{lemma}\label{base-disjointness-preservation} Let $1\le p\ne q<\infty$. Let $X,Y$ be two $\mathcal BL_pL_q$-Banach lattices. If either $p\ne 2, q\ne 1$ and  $p\not\in [q,2]$ (in the case  where $q<2$), or if $X$ is a $\mathcal BL_pL_q^{\ge 2}$-Banach lattice, then every linear isometry from $X$   to $Y$ preserves base-disjointness. \end{lemma}

\begin{proof} We may assume w.l.o.g. that $Y$ is a $L_p(L_q)$-space. It is easy to see that the analogue of Lemma \ref{re-dp-implies-comp-dp} is true, where we replace the words `disjointness preserving' by `base-disjointness preserving'. For this reason we may restrict our attention to the real Banach lattice case.

Assume that $u,v$ are two norm one base-disjoint elements in $X$. Then $Tu,Tv$ form a 
 $\ell_2^p$-basis in $Y$.  If  $p\ne 2, q\ne 1$ and  $p\not\in [q,2]$ (in the case  where $q<2$),  it results from \cite[Prop. 5.6]{GR} that $Tu,Tv$ are base-disjoint. 
 
 In the remaining cases where $p=2$, or $q=1$ or $q<p<2$, if now $X$ is a $\mathcal BL_pL_q^{\ge 2}$-Banach lattice we shall consider two particular subcases.
 
The first particular case is when the base-band $B_v$ generated by $v$ and the (ordinary) band generated by $v$ coincide on no sub-base-band of $v$ (in other words, $B_v\ominus\band[v]$ and $v$ have the same base-support). In this case we may find in $B_v\ominus\band[v]$  an element $w$ such that $N(w)=N(v)$. Then $(v,w)$ form a $\ell_q^2$-basis in $X$ and $(u,x)$  form a $\ell_p^2$-basis for every norm-one element $x\in\mathrm{span}\,[v,w]$. The same is true for $Tu,Tv,Tw$ in $Y$, since $T$ is an isometry. Hence by [GR, prop. 5.7], $Tu$ and $Tx$ are base-disjoint in $Y$ for any $x\in\mathrm{span}\,[v,w]$, in particular for $x=v$.

A second  particular subcase occurs when $\band[v]=B_v$. In this case  by Lemma \ref{fiber-dlmension} we can find disjoint elements  $v_1,v_2$ in $B_v$, each of them generating $B_v$ as base-band. We may w.l.o.g. assume that $v_1+v_2=v$: indeed if not replace $v_1$ by $v'_1=P_{v_1}v$ (the component of $v$ in the band generated by $v_1$) and $v_2$ by $v'_2=v-v'_1$; then $v'_1$ and $v'_2$ are disjoint, $\band[v'_1]=\band[v_1]$ and $\band[v'_2]\supset \band[v_2]$ which imply that each of $v'_1$, $v'_2$ generates $B_v$ as base-band. Now we prove separately that $Tu,Tv_1$ and $Tu,Tv_2$ are base disjoint, this will imply that $Tu$ and $Tv=Tv_1+Tv_2$ are also base-disjoint.  For proving that $Tu$ and $Tv_1$ are base-disjoint, we find $w_2$ with $\band[w_2]=\band[v_2]$ such that $N(w_2)=N(v_1)$ (simply set $w_2={N(v_1)\over N(v_2)}v_2$). Then $v_1,w_2$ are disjoint and span $\ell_q^2$, and $u$ is base-disjoint of any element $x\in\mathrm{span}\,[v_1,w_2]$ so that we can conclude as in the first case. The same reasoning works for proving that $Tu$ and $Tv_2$ are base-disjoint.

In the general case we can split $v$ as a sum $v'+v''$ of two base disjoint components, the first one falling into the first subcase above, the second one in the second subcase. Simply split $B_v$ into the base-band $B'_v$ generated by $B_v\ominus \band[v]$ and the complementary one $B''_v=B_v\ominus B'_v$, and define $v',v''$ as being the corresponding components of $v$. By the preceding reasonings, $Tu$ is base-disjoint from both $Tv'$ and $Tv''$ and thus from $Tv=Tv'+Tv''$.
\end{proof}

\begin{proposition} Let $1\le p\ne q<\infty$ with $q\ne 2$ and $q\not\in [p,2]$ if $p\le 2$. Then if $p>1$ the  class $\mathcal BL_pL_q^{\ge k}$ has property DPIU for all $k\ge 2$ while the class $\mathcal BL_1L_q^{\ge k}$ has DPIU for all $k\ge 3$. Moreover if $p>2,q>1$ or $q>2,p\ne1,2$, the class $\mathcal BL_pL_q$  has property DPIU. 
\end{proposition}

\begin{proof} We show that a linear isometry $T$ from a member  $X$ of the class under consideration 
($\mathcal BL_pL_q^{\ge k}$  with $k=1,2$ or $3$, depending on $p,q$) into a $L_p(L_q)$-Banach lattice $Y$ is disjointness preserving. We consider disjoint elements $u,v$ of $X$ and want to show that their images $Tu,Tv$ are disjoint. By Lemma \ref{fiber-dlmension} we may find disjoint elements $(x_1,\dots x_k)$ having the same base-support as $|u|+|v|$ and spanning $\ell_q^k$. There is a  separable sublattice $X_0$ of the base-band generated by $u$ and $v$ in $X$ which contains the pair $\{u,v\}$ and the system  $\{x_1,\dots x_k\} $and is itself a  $\mathcal BL_pL_q$-Banach lattice : this can be seen by an application of the downwards L\"owenheim-Skolem theorem (or directly like in the proof of Lemma \ref{sep-fib-atl}).  By Lemma \ref{fiber-dlmension} again, $X_0$ is still a $\mathcal BL_pL_q^{\ge k}$-Banach lattice ($k$=1,2,3). Thus we may restrict our attention to the case where $X$ is separable. Then $X$ can be represented as a $p$-direct sum $\bigoplus_n L_p(\Omega_n,\A_n,\mu_n; L_q(\Omega'_n,\A'_n,\mu'_n))$ where for each $n$, $L_p(\Omega_n,\A_n,\mu_n)$ and $L_q(\Omega'_n,\A'_n,\mu'_n)$ are separable. The hypothesis on $X$ means that for all $n$, $\dim L_q(\Omega'_n,\A'_n,\mu'_n)\ge k$, $k=1,2$ or $3$ depending on $p,q$.

Let $T_n$ be the restriction of the isometry $T$ to the factor $X_n=L_p(\Omega_n,\A_n,\mu_n; L_q(S_n,\Sigma_n,\nu_n))$. By Proposition \ref{DP-for-LpLq-isometries} if $\dim L_q(\Omega'_n,\A'_n,\mu'_n)\ge 2$  (or $\ge 3$ in the case $p=1$) this isometry is disjointness preserving. On the other hand if $k\ge 2$ the isometry $T$ is base-disjointness preserving by Lemma \ref{base-disjointness-preservation}. Since the factors  $X_n$ are base-disjoint and $T$ preserves base-disjointness, the images $T(X_n)$ are pairwise base-disjoint  and $T$ is disjointness-preserving. This proves the proposition in the cases $k=2,3$.

In the case $k=1$, if $p>2,q>1$ or $q>2,p\ne1,2$ then, again by Lemma \ref{base-disjointness-preservation}, $T$ preserves base-disjointness even if in the factor $X_n$, the fiber $L_q(\Omega'_n,\A'_n,\mu'_n)$ is 1-dimensional. In this case $X_n=L_p(\Omega_n,\A_n,\mu_n)$, and base-disjointess in $X_n$ is the same as disjointness. Hence the isometry $T_n$ is also disjointness preserving and we conclude as before.
\end{proof}

\begin{proposition}\label{Axiomatizability}
The class $\mathcal BL_pL_q^{\ge n}$-Banach lattices is axiomatizable.
\end{proposition}

\begin{proof} 
The class  of $\mathcal BL_pL_q$-Banach lattices itself is axiomatizable  \cite{HR}.

If for some $k<n$, $X\in \mathcal BL_pL_q$ contains a base-band $Y$ isomorphic to $L_p(\ell_q^k)$ then for every ultrafilter $\U$, the $\mathcal BL_pL_q$-Banach lattice $X_\U$ contains the base-band $Y_\U$ which isomorphic to $L_p(\ell_q^k)_\U$, which is itself isomorphic to $(L_p)_\U(\ell_q^k)$. Thus by Lemma 
\ref{fiber-dlmension} if $X$ does not belong to $\mathcal BL_pL_q^{\ge n}$, neither does $X_\U$, and the class $\mathcal BL_pL_q^{\ge n}$ is closed under ultraroots.

Let us show that the class $\mathcal BL_pL_q^{\ge n}$ is closed under ultraproducts. Let $(X_\alpha)_{\alpha\in \mathbb A}$ be a family of members of  $BL_pL_q^{\ge n}$, and $X=\Pi_\U X_\alpha$ an ultraproduct of this family. We may consider each $X_\alpha$ as a band in a space $\Lambda_\alpha=E_\alpha(F_\alpha)$, where $E_\alpha$ is an $L_p$-space and $F_\alpha$ a $L_q$-space. Let  $N_\alpha:\Lambda_\alpha\to E_\alpha$ be the corresponding  random norm. Then by \cite{LR}, $X$ may be represented as a band in $E(F)$, where $E=\Pi_\U X_\alpha$ and $F$ is some $L_q$ space, in such a way that the random norm $N:E(F)\to E$ coincides on $X$ with the ultraproduct map $\Pi_\U N_\alpha$. 

Let $Y$ be a nontrivial base-band in $X$, we want to prove that $Y$ contains a sublattice isomorphic to $\ell_q^n$. We may assume that $Y$ is the base-band generated by a single nonzero element $\xi$. Let $\xi=[x_\alpha]_\U$ be a representation of $\xi$ by a bounded  family  in $\prod X_\alpha$, consisting of nonzero elements. Let $Y_\alpha$ be the base-band generated by $x_\alpha$ in $X_\alpha$. By Lemma \ref{fiber-dlmension} we can find a system $(x_1^\alpha,\dots , x_n^\alpha)$ of disjoint elements of $X_\alpha$ generating a $\ell_q^n$ sublattice, each $x^j_\alpha$ generating $X_\alpha$ as base-band. In fact we may find such a system with $N(x_1^\alpha)=\dots N(x_n^\alpha)=N(x_\alpha)$. 
Now set $\xi_j=[x_j^\alpha]_\U$, $j=1,\dots n$, then $N(\xi_j)=[N(x_j^\alpha)]_\U=[N(x_\alpha)]_\U=N(x)$, which shows that the $\xi_j$ have the same base-support as $\xi$ and thus belong to the base-band $Y$, the $\xi_j$ are disjoint and span a $\ell_q^n$-sublattice of $Y$.
 \end{proof}
 
Let us finally give a summary of our present knowledge about Banach axiomatizability for classes $\mathcal BL_pL_q$:

\begin{theorem} 
\parindent=0pt{
a) If $p,q\in ]1,2[\cup ]2,+\infty[$, $p\ne q$  then the class $\mathcal BL_pL_q$ is axiomatizable in the language of Banach spaces; so are the classes $BL_pL_q^{\ge n}$ for all $n\ge 1$.

b) If $1<p<\infty$ then the class $\mathcal BL_pL_1^{\ge 2}$, and more generally any class $\mathcal BL_pL_1^{\ge n}$, $n\ge 2$, is axiomatizable in the language of Banach spaces.

c) If $q\ne 2$ then the class $\mathcal BL_2L_q^{\ge 2}$, and more generally any class $\mathcal BL_2L_q^{\ge n}$, $n\ge 2$, is axiomatizable in the language of Banach spaces.

d) If $q>2$ the class $\mathcal BL_1L_q^{\ge 3}$, and more generally any class $\mathcal BL_1L_q^{\ge n}$, $n\ge 3$, is axiomatizable in the language of Banach spaces.

e) If $1<p<\infty $ the class $\mathcal BL_pL_2$ is axiomatizable in the language of Banach spaces.}
\end{theorem}

\begin{proof}

\parindent=0pt
{a) results from  Prop. \ref{Axiomatizability} and DPIU property when $p$ or $q>2$, and is obtained by duality when $1<p,q<2$.

b)   c)  and d) result from DPIU property and  Prop. \ref{Axiomatizability};

e) results from the analysis of contractive projections in $L_p(H)$, where $H$ is an Hilbert space, see \cite{R}.}
\end{proof}

The main question  remaining open in the context of $\mathcal BL_pL_q$-spaces is:

\begin{question}
 Is the class $\mathcal BL_1L_q$ axiomatizable in the language of Banach spaces, for $1<q\le 2$?
\end{question}


\begin{thebibliography}{99}

\bibitem{AA} Y. A. Abramovich and C. D. Aliprantis, {\it An invitation to Operator theory}. Graduate Studies in Mathematics  {\bf 50}, American Mathematical Society, Providence, RI (2002). 

\bibitem{AB} C. D. Aliprantis and O. Burkinshaw, {\it Positive operators}. Springer Verlag, Dordrecht (2006).

\bibitem{BBHU} I. Ben Yaacov, A. Berenstein, C. W. Henson, and A. Usvyatsov, {\it Model theory for metric structures}. Model Theory with Applications to Algebra and Analysis, Vol. 2,  315--427, London Math. Soc. Lecture Note Ser. {\bf 350}, Cambridge Univ. Press (2008).

\bibitem{BB} K. Boulabiar and G. Buskes, {\it Polar decomposition of disjointness preserving operators}. Proc. Amer. Math. Soc. {\bf 132} (2004), no. 3, 799-806.

\bibitem{CZ} W. J. Claas and  A. C. Zaanen, {\it Orlicz lattices}. 
Special issue dedicated to W\l adys\l aw Orlicz on the occasion of his seventy-fifth birthday. 
Comment. Math. Special Issue {\bf 1} (1978), 77-93. 

\bibitem{F} D. H. Fremlin, {\it Measure Theory, vol.3: Measure algebras}. Torres Fremlin, Colchester (2002).

\bibitem{GH}  J. J. Grobler and C. B. Huijsmans, {\it Disjointness preserving operators on complex Riesz spaces}. Positivity {\bf 1} (1997), 155-164.

\bibitem{GR}  S. Guerre and Y. Raynaud, {\it Sur les isom\'etries de $L_p(X)$ et le th\'eor\`eme ergodique vectoriel}.  Can. J. Math. {\bf 140} (1988),  360-391.

\bibitem{H}  S. Heinrich, {\it Ultraproducts in Banach Spaces Theory}.  J. Reine
Angew. Math. {\bf 313} (1980), 72-104.

\bibitem{CWH} C. W. Henson and L.C. Moore Jr., {\it Nonstandard analysis and the theory of Banach spaces}. Nonstandard analysis--recent developments (Victoria, B.C., 1980), 27-112,  Lecture Notes in Math. {\bf 983}, Springer, Berlin, 1983.

\bibitem{HI} C. W. Henson and J. Iovino, {\it Ultraproducts in Analysis}. Analysis
and Logic (Mons 1997), 1--113, London Math. Soc. Lecture Note Ser. {\bf 262}, Cambridge Univ. Press (2003).

\bibitem{HR} C. W. Henson and Y. Raynaud, {\it On the theory of $L_p(L_q)$ Banach lattices}.  Positivity {\bf 11} (2007), 201-230.

\bibitem{HR-11} C. W. Henson and Y. Raynaud, {\it Quantifier elimination in the theory of $L_p(L_q)$-Banach lattices}. J. Log. Anal. {\bf 3} (2011), Paper 11.

\bibitem{JKP}  J. E. Jamison, A. Kami\'nska, and Pei-Kee Lin, {\it Isometries of Musielak-Orlicz spaces II}.  Studia Math. {\bf 104} (1993),  75-89.

\bibitem{JLS}  W. B. Johnson, J. Lindenstrauss, and G. Schechtman, {\it Banach spaces determined by their uniform structures}. Geom. Funct. Anal. {\bf 6} (1996), 430-470.

\bibitem{L} H. E. Lacey, {\it Local unconditional structure in Banach spaces}. Banach spaces of analytic functions, Lecture Notes in Math. {\bf 604} (1977), 44-56.

\bibitem{La} J. Lamperti, {\it On the isometries of certain function-spaces}. Pacific J. Math. {\bf 8} (1958) 459-466. 

\bibitem{LR}  M. Levy and Y. Raynaud,  {\it Ultrapuissances des espaces $L^p(L^q)$}.  C. R. Acad. Sci. Paris S\'er. I Math. {\bf 299} (1984), no. 3, 81-84. 

\bibitem{LT} J. Lindenstrauss and L. Tzafriri, {\it Classical Banach Spaces II: Function
Spaces}. Ergebnisse der Math. {\bf 97}, Springer Verlag (1979).

\bibitem{M76} M. Meyer, {\it Le stabilisateur d'un espace vectoriel r\'eticul\'e}, C. R. Acad. Sci. Paris S\'er. A {\bf 283} (1976) 249-250.

\bibitem{M81} M. Meyer, {\it Les homomorphismes d'espaces vectoriels r\'eticul\'es complexes}, C. R. Acad. Sci. Paris S\'er. I Math. {\bf 292} (1981) 793-796.

\bibitem{MN} P.~Meyer-Nieberg, \emph{Banach Lattices}. Springer-Verlag, Berlin Heidelberg 1991.

\bibitem{P}  L. P. Poitevin,  \emph{ Model Theory of Nakano Spaces}. Ph.D. thesis, University of Illinois at Urbana-Champaign, 2006.

\bibitem{PR}  L. P. Poitevin and Y. Raynaud, {\it Ranges of positive contractive projections in Nakano spaces.}  Indag. Math. {\bf 19} (2008), 441-464.

\bibitem{R}  Y. Raynaud, {\it The range of a contractive projection in $L_p(H)$.}  Rev. Mat.
Complut. {\bf 17} (2004), 485-512.

\bibitem{Ri} M. J. Rivera, {\it On the classes of $\mathcal L^\lambda$, quasi-$\mathcal L^E$ and $\mathcal L^{\lambda,g} spaces$}. Proc. Amer. Math. Soc. {\bf 133} (2005), no. 7, 2035-2044.

\bibitem{W81} W. Wnuk, {\it On a representation theorem for convex Orlicz lattices}. Bull. Acad. Polon. Sci. S\'er. Sci. Math. {\bf 28} (1980), no. 3-4, 131-136 (1981).

\bibitem{W84} W. Wnuk, {\it Representations of Orlicz lattices}. Dissertationes Math. (Rozprawy Mat.) {\bf 235} (1984). 



\end{thebibliography}
\end{document}